\documentclass[11pt]{article}
\usepackage{amsmath,amsfonts,amssymb, amsthm,enumerate,graphicx}
\usepackage{tikz,authblk}
\usepackage{subfig}
\newtheorem{theorem} {{\textsf{Theorem}}}[section]
\newtheorem{proposition}[theorem]{{\textsf{Proposition}}}

\newtheorem{definition}[theorem]{{\textsf{Definition}}}
\newtheorem{remark}[theorem]{{\textsf{Remark}}}

\newtheorem{lemma}[theorem]{{\textsf{Lemma}}}
\newtheorem{problem}[theorem]{{\textsf{Problem}}}

\newcommand{\Star}{\mbox{\upshape st}\,}
\newcommand{\lk}{\mbox{\upshape lk}\,}
\newcommand{\RR}{\mathbb{R}}
\newcommand{\ZZ}{\mathbb{Z}}
\newcommand{\Ss}{\mathcal{S}}

\title{Three-dimensional normal pseudomanifolds with relatively few edges}
\author{Biplab Basak\\
\small Department of Mathematics\\
\small Indian Institute of Technology Delhi, Hauz Khas, New Delhi 110016, India\\
\small \texttt{biplab@iitd.ac.in}\\ 
\and Ed Swartz\\
\small Department of Mathematics\\
\small Cornell University, Ithaca NY, 14853-4201, USA,\\
\small \texttt{ebs22@cornell.edu}}
\begin{document}
\maketitle

\begin{abstract}
  Let $\Delta$ be a $d$-dimensional normal pseudomanifold, $d \ge 3.$  A relative lower bound for the number of edges in $\Delta$  is that $g_2$ of $\Delta$ is at least $g_2$ of the link of any vertex. When this inequality is sharp  $\Delta$ has {\bf relatively minimal $g_2$.}  For example, whenever the one-skeleton of $\Delta$ equals the one-skeleton of the star of a vertex, then $\Delta$ has relatively minimal $g_2.$  Subdividing a facet in such an example also gives a complex with relatively minimal $g_2.$  We prove that in dimension three these are the only examples.  As an application  we  determine the combinatorial and topological type of $3$-dimensional $\Delta$ with relatively minimal $g_2$ whenever $\Delta$ has two or fewer singularities.  The topological type of any such complex is a pseudocompression body, a pseudomanifold version of a compression body.  
 
 Complete combinatorial descriptions of $\Delta$ with $g_2(\Delta) \le 2$ are due to Kalai \cite{Kalai} $(g_2=0)$, Nevo and Novinsky \cite{NevoNovinsky} $(g_2=1)$ and Zheng \cite{Zheng} $(g_2=2).$  In all three cases $\Delta$ is the boundary of a simplicial polytope.  Zheng observed that for all $d \ge 0$ there are triangulations of $S^d \ast \RR P^2$ with $g_2=3.$  She asked if this is the only nonspherical topology possible for $g_2(\Delta)=3.$  As another application of relatively minimal $g_2$ we give an affirmative answer when $\Delta$ is $3$-dimensional.

\end{abstract}

\section{Introduction}

 The fundamental enumerative invariant of a finite $d$-dimensional simplicial complex is its $f$-vector $(f_{-1},f_0, \dots, f_d),$ where   $f_i$ is the number of $i$-dimensional faces for $-1\leq i \leq d$. For all complexes the empty set is the only $(-1)$-dimensional face.  So $f_{-1}=1.$ The study of the interplay between the $f$-vector of a complex and its topology goes back to the dawn of algebraic topology and the invariance of the Euler characteristic.   Inequalities involving the $f_i$ have proven to be a very interesting and challenging approach to analyzing the relationship between  the $f$-vector of a complex and its topology.  One of the first big steps in this direction was Walkup's 1970 Acta paper, ``The lower bound theorem for $3$- and $4$-manifolds'' \cite{Walkup}.  For any $d$-dimensional simplicial complex $g_2$ is defined to be $f_1 - (d+1) f_0 + \binom{d+2}{2}.$ Walkup proved that if $\Delta$ is a  triangulation of a  closed connected $3$-manifold, then $g_2 \ge 0.$ He also proved that if $g_2=0,$ then $\Delta$ is a stacked sphere. Later, Barnette extended these results to manifolds of higher dimension \cite{Barnette1}, \cite{Barnette2}. In the 80's Kalai  established this theorem for the class of normal pseudomanifolds of dimension at least three whose two-dimensional links were spheres \cite{Kalai}.    Fogelsanger's thesis made it possible to remove the restriction on two-dimensional links  \cite{Fogelsanger}. Another proof of the nonnegativity of $g_2$ is due to Gromov \cite[pages 211--212]{Gromov} with a detailed exposition by Bagchi and Datta \cite{BagchiDatta}.
 
One way to see the importance of $g_2$ is the main result of \cite{Swartz2008}.  Up to scaling,  $g_2$ is the unique affine invariant of $f_0$ and $f_1$ with the following properties:  Fix $d \ge 2$ and $\Gamma \ge 0$  and let $\mathcal{T}(d,\Gamma)$ be the collection of all triangulations $\Delta$ of PL-manifolds whose dimension is bounded by $d$ and $g_2(\Delta) \le \Gamma.$  Then  $g_2(\partial \Delta^{d+1})=0,~ g_2(\Delta) \ge 0$ for all $\Delta \in \mathcal{T}(d,\Gamma),$ and $\mathcal{T}(d,\Gamma)$ is an infinite set but only has finitely many homeomorphism types.

  Kalai's paper included a relative lower bound based on the links of the complex.  With the addition of Fogelsanger's thesis, Kalai's relative lower bound for $g_2$ says that if $\Delta$ is a normal $d$-pseudomanifold with $d \ge 3,$ then $g_2(\Delta) \ge g_2(\lk v)$ for any vertex $v$ of $\Delta.$ This opens the possibility of establishing lower bound theorems which depend on the singularities of the pseudomanifold.  For $3$-dimensional normal pseudomanifolds this is particularly attractive as all singular links occur at vertices, are connected compact surfaces, and the set of singularities is a homeomorphism invariant.  Furthermore, for any connected compact surface $K,$  $g_2(K)$ is always $3 b_1(K),$ where $b_1(K)$ is the first Betti number of $K$ with $\ZZ/2\ZZ$ coefficients.  The question becomes, given a (finite) multiset of nonspherical connected compact surfaces $\Ss=\{K_1, \dots, K_n\}$ what is the minimum of $g_2$ over all triangulations of $3$-dimensional normal pseudomanifolds whose multiset of singular vertex links is $\Ss.$  The relative lower bound mentioned above implies that it is at least $\max_i 3  b_1(K_i).$ This lower bound is sharp for over $500$ of the $1578$ distinct multisets of singularities which occur in all triangulations of $3$-dimensional normal pseudomanifolds with $10$ or fewer vertices \cite{AkhShiSwa}.
  
One  place to start is to study $\Delta$ where $g_2(\Delta)=g_2(\lk v)$ for some vertex $v \in \Delta.$  Such complexes have {\bf relatively minimal} $g_2.$ What are the possible topological/combinatorial types of complexes with relatively minimal $g_2?$  If the $1$-skeleton of $\Delta$ is the $1$-skeleton of the star of a vertex, then $\Delta$ has relatively minimal $g_2.$  Facet subdivisions of such complexes also have relatively minimal $g_2.$  Theorem \ref{no outside vertices} says that in dimension three all complexes with relatively minimal $g_2$ have this structure.

For which multisets of singularities $\Ss$ does there exist $\Delta$ such that the multiset of singularities of $\Delta$ is $\Ss$ and $\Delta$ has relatively minimal $g_2? $  We give a complete answer to this question in dimension three when  $|\Ss|\le2.$  In this situation the topological type is always a pseudocompression body, a pseudomanifold version of compression bodies (Theorems \ref{two singularities - combinatorics} and \ref{two singularities - homeomorphism}).
  
  Since  Kalai's work there has been an effort to classify normal pseudomanifolds with small $g_2.$  Nevo and Novinsky did this for $g_2=1$ \cite{NevoNovinsky} and Zheng for $g_2=2$ \cite{Zheng}.  In both cases all the complexes are polytopal.  Zheng observed that for all $d \ge 3$ there are triangulations of $S^{d-3} \ast \RR P^2$ with $g_2=3.$  She asked if this is the only topological type possible other than spheres when $g_2=3.$  We prove that this is true for $d=3$ and give a combinatorial description of all such triangulations (Theorem \ref{g2=3 complexes}).

The next two sections  establish the notation and constructions that will be needed.  The only   idea that is not well known is the use of folding a pseudomanifold.   This involves identifying two facets which share one or two vertices.  Section \ref{rigidity} introduces Kalai's techniques and results concerning the use of rigidity in face enumeration.  This is followed by a general study of complexes with relatively minimal $g_2.$  The last two sections state and prove the main theorems discussed above.

\section{Notation} \label{notation}

All of our simplicial complexes are finite and abstract - a subset of the power set of a finite set which is closed under subsets.  
The vertex set of a simplicial complex $\Delta$ is $V(\Delta),$ or just $V$ if no confusion is possible.  Similarly, the edge set of $\Delta$ is $E(\Delta)$  or $E,$  and the {\bf graph} of $\Delta$ is  $G(\Delta)$ or  $G$ and is the one-skeleton of $\Delta.$ Faces of $\Delta$ will usually be written by  concatenating the vertices of the face.  So $xy$ is an edge and $abcd$ is a tetrahedron.  A sequence $(v_1, \dots, v_n)$ of  vertices in $\Delta$ is a {\bf path} if for all $1 \le i \le n-1, ~v_i v_{i+1}$ is an edge.    We will use $P(v_1, \dots, v_n)$ to indicate $(v_1, \dots, v_n)$ is a path. A {\bf circle} in $\Delta$ is a path $P(v_1, \dots, v_n)$ such that $v_1=v_n$ and $v_i \neq v_j$ for all other pairs $i \neq j.$   Circles will be denoted by $C(v_1, \dots, v_n).$   A {\bf graph cone point} of $\Delta$ is a vertex which has every other vertex as a neighbor.  As usual, $|\Delta|$ will represent the topological type of a geometric realization of $\Delta.$  When we say that $\Delta$ has a topological property,  we are actually saying that $|\Delta|$ has the stated property.

The {\bf link} of a face $\sigma$ of $\Delta$ is 
$$ \lk (\sigma, \Delta) = \{\tau \in \Delta: \sigma \cap \tau = \emptyset, \mbox{ and } \sigma \cup \tau \in \Delta\}.$$
If no confusion is possible, then we use $\lk \sigma.$  The {\bf star} of a face will always be the closed star and is defined by
$$\Star (\sigma, \Delta) = \{ \tau \in \Delta: \sigma \cup \tau \in \Delta\}.$$
As with the link, $\Star (\sigma, \Delta)$ might be shortened to $\Star \sigma.$  

Maximal faces of a simplicial complex are called {\bf facets.}  When all facets have the same dimension the complex is called {\bf pure.}  {\bf Pseudomanifolds} are pure simplicial complexes in which every codimension-one face is contained in exactly two facets and a $d$-pseudomanifold is a $d$-dimensional pseudomanifold.  If every codimension-one face of $\Delta$ is contained in one or two facets, and  the collection of the codimension-one faces contained in exactly one facet forms a pseudomanifold, then $\Delta$ is a {\bf pseudomanifold with boundary.}  In this case the {\bf boundary} of $\Delta$ is $\partial \Delta,$ the pseudomanifold whose facets are the codimension-one faces contained in exactly one facet of $\Delta.$  

A pseudomanifold in which the links of all faces of codimension-two or more are connected is called {\bf normal.}  Any triangulation of $S^1$ or $S^0$ is also a normal pseudomanifold.  The definitions imply that the link of a face of a normal pseudomanifold is a normal pseudomanifold.  Two-dimensional normal pseudomanifolds are triangulations of connected compact surfaces.  So the vertex links of $3$-dimensional normal pseudomanifolds are  triangulations of connected compact surfaces.  Vertices whose links are spheres are {\bf nonsingular}.  All other vertices are {\bf singular.}  For a normal three-dimensional pseudomanifold $\Delta$ its {\bf multiset of singularities} is the multiset $\Ss(\Delta)$ of nonspherical vertex links.  

The main enumerative invariant of a simplicial complex is its $f$-vector.  The number of $i$-dimensional faces of $\Delta$ is $f_i$ and for a $d$-dimensional complex the $f$-vector is $(f_{-1}, f_0, \dots, f_d).$  The empty set is the only $(-1)$-dimensional face. The $h$-vector is a linear transformation of the $f$-vector. For a $d$-dimensional complex the $h$-vector is  $(h_0, \dots, h_{d+1})$ and is defined by
$$ h_i = \sum^i_{j=0} (-1)^{i-j} \binom{d+1-j}{i-j} f_{j-1}.$$  Once $h_i$ is defined, $g_i$ is given by $h_i - h_{i-1}.$  Thus $g_2=f_1 - (d+1) f_0 + \binom{d+2}{2}.$  

One of the primary questions under consideration is the relationship between $\Ss(\Delta)$ and $g_2(\Delta).$  Let $\Ss$ be a multiset of connected compact surfaces.  We define

$$\Gamma(\Ss) = \min_{\Ss(\Delta) = \Ss} g_2(\Delta).$$

 Compression bodies (cf. \cite{BF,Schultens}) are one of the building blocks of $3$-manifolds with boundary.  A {\bf compression body} is a $3$-manifold with boundary that is built by the following procedure. Start with a  compact surface  $K$ with components $K_1, \dots, K_l.$  Then $K \times [0,1]$ is a three-manifold with boundary.  Now connect $1$-handles to $K \times \{1\}$ to get a new manifold $C.$   The resulting manifold is a {\bf compression body}.  The boundary of $C$  consists of the disjoint union of $K_1 \times \{0\}, \dots, K_l \times \{0\}$ and whatever boundary is formed by $K \times \{1\}$ and the attached $1$-handles.  The last part is called the {\bf top}  of $C.$ Compression bodies can be used as building blocks for three-manifolds with boundary. 

\begin{definition}
  A {\bf pseudocompression body} is a connected compression body with each boundary component coned off.
\end{definition}

Some examples of pseudocompression bodies are $S^3,$ handlebodies with their boundaries coned off and suspensions of surfaces.


Instead of using genus to specify the topological type of a connected compact surface $|\Delta|$ we  use  $b_1(\Delta) = \dim_{\ZZ/2\ZZ}  H_1(\Delta; \ZZ/2\ZZ)$  and its orientability if it is not clear from the parity of $b_1(\Delta).$ If $|\Delta|$ is a connected compact surface, then $f_1 =3 f_0 -3(2-b_1(\Delta)).$  Hence $g_2(\Delta) = 3~b_1(\Delta)$ and only depends on  the topology of the surface.

\section{Constructions} \label{constructions}

The one-vertex suspension of a simplicial complex $\Delta$ is a way of producing a complex homeomorphic to $\Sigma \Delta,$ the suspension of $\Delta,$ using only one extra vertex instead of the usual two extra vertices.   One-vertex suspension concept was first introduced in \cite{BagchiDatta1},  and later in \cite{BL}, the authors renamed the concept as one-point suspension. Let $v$ be a vertex of $\Delta.$  Form a new complex $\Sigma_v \Delta$, which is called the {\bf one-vertex suspension of} $\Delta$ as follows.  First remove $v$ and all of its incident faces from $\Delta.$  Now add two new vertices, $x$ and $y$ and an edge between them.  Set the link of the edge $xy$ in $\Sigma_v \Delta$ to be equal to the link of $v$ in $\Delta.$  Finally, add all faces of the form $x\sigma$ and $y\sigma$ for all faces $\sigma$ of $\Delta$ not in the star of $v.$  To see that $\Sigma_v \Delta$ is homemorphic to the usual suspension of $\Delta$ perform a stellar subdivision on the edge $xy.$  The {\bf suspension points} of $\Sigma_v \Delta$ are $x$ and $y.$  Observe that the links of $x$ and $y$ are simplicially isomorphic to $\Delta.$

Let $n$ be the number of vertices of $\Delta$ which are {\it not} neighbors of $v.$  Direct computation shows that $g_2(\Sigma_v\Delta) = g_2(\Delta) + n.$  In particular, if $v$ is a graph cone point of  $\Delta$ (i.e., $v$ has every other vertex as a neighbor in $\Delta$), then $g_2(\Delta) = g_2(\Sigma_v \Delta).$ 

Let $\sigma_1,\sigma_2$ be two faces of  $\Delta$. A bijection $\psi:\sigma_1\to \sigma_2$ is said to be {\bf admissible}  (cf. \cite{BagchiDatta}) if for all vertices $x \in \sigma_1$ the edge distance between $x$ and $\psi(x)$ in the  graph of $\Delta$ is at least three. If $\psi$ is an admissible bijection between $\sigma_1$ and $\sigma_2$ then we can form a new simplicial complex by identifying  all faces $\rho_1 \subseteq \sigma_1, \rho_2 \subseteq \sigma_2$ such that $\psi$ maps $\rho_1$ bijectively onto $\rho_2.$ When $\sigma_1$ and $\sigma_2$ are facets we define $\Delta^{\psi}$ to be the simplicial complex described above with the identified facet removed.  When $\sigma_1$ and $\sigma_2$ are also in the same connected component of $\Delta$ we say $\Delta^\psi$ is obtained from $\Delta$ by a {\bf handle addition} (cf. \cite{Walkup}).  If instead $\sigma_1$ and $\sigma_2$ are in distinct components, then we say $\Delta^\psi$ is a {\bf  connected sum} and frequently (but not always) write $\Delta^\psi = \Delta_1 ~\#_\psi ~\Delta_2,$ where $\Delta_1$ and $\Delta_2$ are two distinct components of $\Delta$ such that $\Delta=\Delta_1 \cup \Delta_2, \sigma_1 \in \Delta_1$ and $\sigma_2 \in \Delta_2.$  One of the simplest connected sum constructions  for a $d$-dimensional complex is  $\Delta \# \partial \Delta^{d+1}.$  Combinatorially, this is the same as removing a facet $\sigma$ of $\Delta$ and replacing it with a cone on the boundary of $\sigma.$ This operation is called a {\bf facet subdivision.}  As long as $d \ge 2,$ subdividing a facet does not change $g_2.$  

A straightforward computation shows that for a $d$-dimensional complex $\Delta$ handle additions satisfy,
\begin{equation} \label{g_2: handles}
g_2(\Delta^\psi) = g_2(\Delta) + \binom{d+2}{2}.
\end{equation}

  \noindent Similarly, for connected sum
\begin{equation} \label{g_2:connected sum}
g_2(\Delta_1 ~\#_\psi~ \Delta_2) = g_2(\Delta_1) + g_2(\Delta_2).
\end{equation}

Let $\bar{x}$ be the vertex  which represents the identified $x$ and $\psi(x)$ in a handle addition or connected sum. In both cases the link of $\bar{x}$ is $\lk (x, \Delta)~ \#_\psi ~ \lk(\psi(x), \Delta).$  For all vertices $y$ which are not in the domain or codomain of $\psi$ the link of $y$ is the same in $\Delta^\psi$ or $\Delta_1 \#_\psi \Delta_2$ as it was in $\Delta.$   In particular, if $\Delta$ was a  triangulation of a $d$-manifold, then the result of handle addition or connected sum is a triangulation of another $d$-manifold.  Similarly, if $\Delta$ was a normal pseudomanifold of dimension at least two, then the output of both operations  is also a normal pseudomanifold of the same dimension as $\Delta.$   

Handle addition and connected sum are a standard part of combinatorial simplicial topology.  But, the operation of {\bf folding} is not  well known.

\begin{definition}
Let $\sigma_1$ and $\sigma_2$ be two facets of a simplicial complex $\Delta$ whose intersection is a single vertex $x.$  A bijection $\psi:\sigma_1 \to \sigma_2$ is {\bf vertex folding admissible} if $\psi(x) = x$ and if for all other vertices $y$ of $\sigma_1$ the only path of length two from $y$ to $\psi(y)$ is $P(y, x, \psi(y)).$  Now, if $\psi$ is folding admissible, then we can again form the complex $\Delta^\psi_v$ obtained by identifying all faces $\rho_1 \subseteq \sigma_1 $ and $\rho_2 \subseteq \sigma_2 $ such that $\psi(\rho_1) = \rho_2,$ and then removing the facet formed by identifying $\sigma_1$ and $\sigma_2.$  In this case we call $\Delta^\psi_x$ a {\bf vertex folding} of $\Delta$  at $x.$   In a similar spirit, $\Delta$ is a {\bf vertex unfolding} of $\Delta^\psi_x.$ 
\end{definition}

An alternative description of a vertex folding at $x$ for a pseudomanifold is sometimes helpful.  First remove $x$ and its incident faces from the complex. This leaves a pseudomanifold with boundary equal to the link of $x$ in $\Delta.$     Identify the two codimension-one faces $\sigma_1 \setminus x$ and $\sigma_2 \setminus x$ via $\psi.$ This is still a pseudomanifold with boundary.  Now cone off the boundary of this complex with $\bar{x}.$  

As with handle additions,  straightforward computations show that if $\Delta^\psi_x$ is obtained from a $d$-dimensional $\Delta$ by a vertex folding at $x$, then

 \begin{equation} \label{folding g2}
g_2(\Delta^\psi_x) = g_2(\Delta)+\binom{d+1}{2},
\end{equation}

\begin{equation} \label{folding linkx}
\lk(\bar{x}, \Delta^\psi_x) = \lk(x,\Delta)^\psi.
\end{equation}

\noindent Now suppose $y \in \sigma_1,~ y \neq x.$ Then $\lk y \cap \lk \psi(y) = \{x\}.$  Relabel $x \in \lk \psi(y)$ as a distinct vertex, say $x'$ and redefine $\psi$ appropriately.  Then

\begin{equation} \label{folding linky}
\lk(\bar{y}, \Delta^\psi_x) = \lk(y,\Delta) \#_\psi~\lk(\psi(y), \Delta).
\end{equation}

Since the link of $\bar{x}$ is formed by handle addition, a complex obtained by a vertex folding at $x$ in  a triangulation of  a  manifold is almost never a  triangulation of a  manifold.  An interesting exception to this is in dimension two.  Depending on the specific identifications, handle addition on a circle can result in either two disjoint circles or one circle.  For an instructive example, one can  check that the six-vertex triangulation of the projective plane is a vertex folding of an eight-vertex stacked two-sphere.  To a certain extent this example is the basis for  edge folding in a $3$-dimensional normal pseudomanifold.

The definition of an edge folding follows the same pattern as the definition for vertex folding.

\begin{definition} 
Let $\sigma_1$ and $\sigma_2$ be facets of a simplicial complex $\Delta$ whose intersection is an edge $uv.$  A bijection $\psi:\sigma_1 \to \sigma_2$ is {\bf edge folding admissible} if $\psi(u)=u,~\psi(v)=v,$ and for all other vertices $y$ of $\sigma_1$ all paths of length two or less from $y$ to $\psi(y)$ go through either $u$ or $v.$ As before, identify all faces $\rho_1 \subseteq \sigma_1$ and $\rho_2 \subseteq \sigma_2$ such that $\psi:\rho_1 \to \rho_2$ is a bijection.  The complex obtained by removing the facet resulting from identifying $\sigma_1$ and $\sigma_2$ is denoted $\Delta^\psi_{uv}$ and is called an {\bf edge folding of $\Delta$ at $uv$.}  As with vertex folding, $\Delta$ is an {\bf edge unfolding} of $\Delta^\psi_{uv}.$  
\end{definition}

Let $\Delta^\psi_{uv}$ be an edge folding of $\Delta$ at $uv$  and assume that $\Delta$ is a normal $d$-pseudomanifold.  As before, let $\bar{x}$ be the image of $x$ in $\Delta^\psi_{uv}.$   Routine checks show that $\Delta^\psi_{uv}$ is a pseudomanifold, 

\begin{equation} \label{edge folding g2} 
g_2(\Delta^\psi_{uv}) = g_2(\Delta)+\binom{d}{2},
\end{equation}
\noindent and
\begin{equation} \label{edge folding u link}
 \lk(u, \Delta^\psi_{uv}) = (\lk(u,\Delta)^\psi_v).
\end{equation}
\noindent   A symmetric formula holds for the link of $v.$  
 For $y \in \sigma_1, y \notin \{u,v\}$ relabel $u$ and $v$ in $\lk(\psi(y),\Delta)$ by $u'$ and $v'$ and redefine $\psi:\sigma_1 \setminus y \to \sigma_2 \setminus \psi(y)$ appropriately.  Then
\begin{equation} \label{edge folding y link}
  \lk(\bar{y}, \Delta^\psi_{uv}) = \lk(y, \Delta) \#_\psi \lk(\psi(y), \Delta).
\end{equation}

When $d=3$ an edge folding of a normal $d$-pseudomanifold may not be normal. Suppose $\sigma_1 = uvxy.$   The link of $uv$ in $\Delta$ is a circle. If the link of $uv$ is given by the circle $C(x,y,a, \dots, b,\psi(x), \psi(y),c, \dots,d),$ then in $\Delta^\psi_{uv}$ the link of $uv$ will be $C(\bar{x}, b, \dots,a, \bar{y}, c \dots, d)$, a circle.   However, if the link of $uv$ is $C(x,y,a, \dots, b,\psi(y), \psi(x),c, \dots,d),$ then in the edge folding at $uv$ the link of $uv$ consists of two disjoint circles $C(\bar{x}, c \dots, d)$ and $C(\bar{y}, b, \dots, a).$ The links of $u$ and $v$ in $\Delta^\psi_{uv}$ are now (single) pinched surfaces.

For the remainder of this section we assume that we are working in dimension three. The terms missing tetrahedron and missing triangle are well known in combinatorial simplicial topology.   A {\bf missing tetrahedron} of $\Delta$ is a quadruple $abcd$ such that $abcd \notin \Delta,$ but $\partial (abcd) \subseteq \Delta.$ A {\bf missing triangle} of $\Delta$ is a triple $abc$ such that $abc \notin \Delta$, but $\partial (abc) \subseteq \Delta.$   In all four constructions, handle addition, connected sum, vertex folding and edge folding, a missing tetrahedron is created when the facet formed by identifying  $\sigma_1$ and $\sigma_2$ is removed.  It is possible to recognize a complex as coming from one of these constructions by examining these missing  tetrahedra.

Suppose $abcd$ is a missing tetrahedron of $\Delta.$  Then for each vertex $x \in \{a,b,c,d\}$ the other three vertices form a missing triangle in the link of $x.$  We say this triangle {\bf separates} the link if its complement consists of exactly two components.  If $\Delta$ was formed via the connected sum or handle addition construction, then all of these triangles separate their corresponding links.     

\begin{lemma} \label{connected sum and handle addition}
Let $\Delta$ be a normal three-dimensional pseudomanifold and suppose $\tau$ is a missing tetrahedron in $\Delta.$  If for every vertex $x \in \tau$  the missing triangle formed by the other three vertices separates the link of $x,$  then $\Delta$ was formed using handle addition or connected sum.  
\end{lemma}

\begin{proof}
It is sufficient to prove that  $|\partial \tau|$ divides a small neighborhood of $|\partial \tau|$ in $|\Delta|$  into exactly two components.   For a detailed proof of why this is sufficient see \cite[Lemma 3.3]{BagchiDatta}.  So let $x$ be a point in  $|\partial \tau|.$  If $x$ is not a vertex, then, since $\Delta$ is a normal pseudomanifold,  small metric balls around $x$ are homeomorphic to a ball and $|\partial \tau|$ divides the ball into two components.  The hypothesis that the missing triangles divide the vertex links into two components implies that  $|\partial \tau|$ also divides small metric balls around vertices into two components.  Compactness of  $|\partial \tau|$ implies that all of these components can be used to glue together  a $\mathbb{Z}/2 \mathbb{Z}$-bundle over  $|\partial \tau|.$  Since  $|\partial \tau|$ is a two-sphere, the bundle must be trivial and hence  $|\partial \tau|$ divides a small neighborhood of itself into exactly two components.   
\end{proof}

How do we recognize a vertex folding?  Suppose $\Delta$ was obtained as a vertex folding at $x$ with $a= \bar{x}$ and $abcd$ the removed facet.  Then for $y \in \{b,c,d\}$ the triangle formed by the other three vertices separates the link of $y$ while $bcd$ does not separate the link of $a.$ In preparation for proving that this property characterizes triangulations which come from the vertex folding operation, we prove the following lemma. 
\begin{lemma}\label{lemma:missingtetra1}
Let $\Delta$ be a 3-dimensional normal pseudomanifold. Let $\tau=abuv$ be a missing tetrahedron such that for $x\in\{a,b\}$, $\partial(\tau \setminus\{x\}])$ separates  $\lk(x,\Delta).$ Then a small neighborhood of  $|\partial(abv)|$ in $|\lk(u,\Delta)|$ is an annulus if and only if a small neighborhood of  $|\partial(abu)|$ in $|\lk(v,\Delta)|$ is an annulus.
\end{lemma}

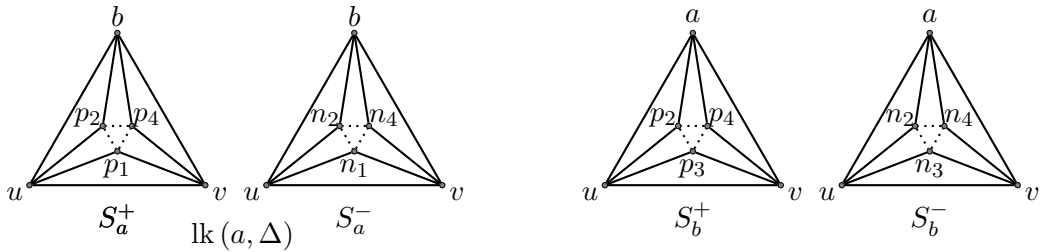
\begin{figure}[ht]
\tikzstyle{ver}=[]
\tikzstyle{vertex}=[circle, draw, fill=black!50, inner sep=0pt, minimum width=2pt]
\tikzstyle{edge} = [draw,thick,-]
\centering
\begin{tikzpicture}[scale=0.45]
\begin{scope}[shift={(-12,0)}]
\foreach \x/\y in {330/d,90/a,210/c}{
    \node[vertex] (\y) at (\x:3){};
}

\foreach \x/\y in {330/v,90/b,210/u}{
\node[ver] () at (\x:3.5){$\y$};
}

\foreach \x/\y in {30/p_1,150/p_2,270/p_3}{
    \node[vertex] (\y) at (\x:0.5){};
}

\foreach \x/\y in {30/p_4,150/p_2,270/p_1}{
\node[ver] () at (\x:1){$\y$};
}

\foreach \x/\y in {d/a,a/c,c/d,a/p_1,d/p_1,a/p_2,c/p_2,c/p_3,d/p_3}{
\path[edge] (\x) -- (\y);}

\foreach \x/\y in {p_1/p_2,p_2/p_3,p_3/p_1}{
\path[edge, dotted] (\x) -- (\y);}
\end{scope}

\begin{scope}[shift={(-5,0)}]
\foreach \x/\y in {330/d,90/a,210/c}{
    \node[vertex] (\y) at (\x:3){};
}

\foreach \x/\y in {330/v,90/b,210/u}{
\node[ver] () at (\x:3.5){$\y$};
}

\foreach \x/\y in {30/n_1,150/n_2,270/n_3}{
    \node[vertex] (\y) at (\x:0.5){};
}

\foreach \x/\y in {30/n_4,150/n_2,270/n_1}{
\node[ver] () at (\x:1){$\y$};
}

\foreach \x/\y in {d/a,a/c,c/d,a/n_1,d/n_1,a/n_2,c/n_2,c/n_3,d/n_3}{
\path[edge] (\x) -- (\y);}

\foreach \x/\y in {n_1/n_2,n_2/n_3,n_3/n_1}{
\path[edge, dotted] (\x) -- (\y);}
\end{scope}

\begin{scope}[shift={(5,0)}]
\foreach \x/\y in {330/d,90/a,210/b}{
    \node[vertex] (\y) at (\x:3){};
}

\foreach \x/\y in {330/v,90/a,210/u}{
\node[ver] () at (\x:3.5){$\y$};
}

\foreach \x/\y in {30/p_4,150/p_2,270/p_3}{
\node[ver] (\y) at (\x:1){$\y$};
    \node[vertex] (\y) at (\x:0.5){};
}

\foreach \x/\y in {d/a,a/b,b/d,a/p_4,d/p_4,a/p_2,b/p_2,b/p_3,d/p_3}{
\path[edge] (\x) -- (\y);}
\foreach \x/\y in {p_4/p_2,p_2/p_3,p_3/p_4}{
\path[edge, dotted] (\x) -- (\y);}
\end{scope}

\begin{scope}[shift={(12,0)}]
\foreach \x/\y in {330/d,90/a,210/b}{
    \node[vertex] (\y) at (\x:3){};
}

\foreach \x/\y in {330/v,90/a,210/u}{
\node[ver] () at (\x:3.5){$\y$};
}

\foreach \x/\y in {30/n_4,150/n_2,270/n_3}{
\node[ver] (\y) at (\x:1){$\y$};
    \node[vertex] (\y) at (\x:0.5){};
}

\foreach \x/\y in {d/a,a/b,b/d,a/n_4,d/n_4,a/n_2,b/n_2,b/n_3,d/n_3}{
\path[edge] (\x) -- (\y);}
\foreach \x/\y in {n_4/n_2,n_2/n_3,n_3/n_4}{
\path[edge, dotted] (\x) -- (\y);}

\end{scope}

\node[ver] () at (-12,-2.5){$S^+_a$ };
\node[ver] () at (-5,-2.5){$S^-_a$ };
\node[ver] () at (-8.3,-3){$\lk(a,\Delta)$ };
\node[ver] () at (-12,-2.5){$S^+_a$ };
\node[ver] () at (5,-2.5){$S^+_b$ };
\node[ver] () at (12,-2.5){$S^-_b$ };

\end{tikzpicture}
\caption{$\lk(a,\Delta)$ and $\lk(b,\Delta)$ in $\Delta$.} \label{fig:lk(a&b)}
\end{figure}

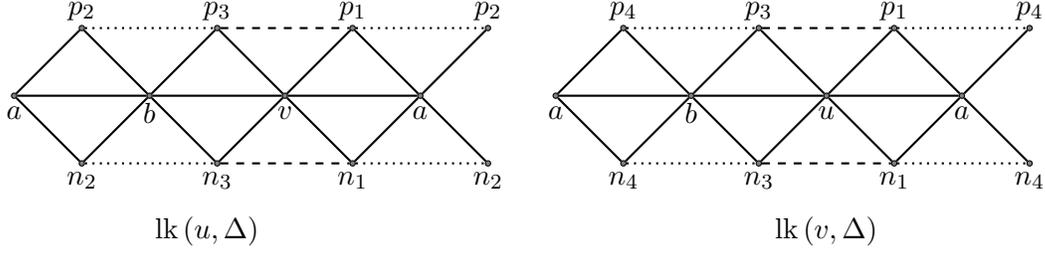
\begin{figure}[ht]
\tikzstyle{ver}=[]
\tikzstyle{vertex}=[circle, draw, fill=black!50, inner sep=0pt, minimum width=2pt]
\tikzstyle{edge} = [draw,thick,-]
\centering
\begin{tikzpicture}[scale=0.45]

\begin{scope}[shift={(-8,0)}]
\foreach \x/\y/\z in {-6/1/b,-2/1/c,2/1/a,6/1/b1}{
\node[vertex] (\z) at (\x,\y){};
}

\foreach \x/\y/\z in {-6/0.5/a,-2/0.5/b,2/0.5/v,6/0.5/a}{
\node[ver] () at (\x,\y){$\z$};
}

\foreach \x/\y/\z in {-4/3/p_2,0/3/p_4,4/3/p_1,8/3/p2}{
\node[vertex] (\z) at (\x,\y){};
}

\foreach \x/\y/\z in {-4/3.5/p_2,0/3.5/p_3,4/3.5/p_1,8/3.5/p_2}{
\node[ver] () at (\x,\y){$\z$};
}

\foreach \x/\y/\z in {-4/-1/n_2,0/-1/n_4,4/-1/n_1,8/-1/q2}{
\node[vertex] (\z) at (\x,\y){};
}

\foreach \x/\y/\z in {-4/-1.5/n_2,0/-1.5/n_3,4/-1.5/n_1,8/-1.5/n_2}{
\node[ver] () at (\x,\y){$\z$};
}

\foreach \x/\y in {b/c,c/a,a/b1,b/p_2,c/p_2,b/n_2,c/n_2,c/p_4,a/p_4,c/n_4,a/n_4,a/p_1,b1/p_1,a/n_1,b1/n_1,
b1/p2,b1/q2}{
\path[edge] (\x) -- (\y);}

\foreach \x/\y in {p_2/p_4,p_1/p2,n_2/n_4,n_1/q2}{
\path[edge, dotted] (\x) -- (\y);}
\path[edge, dashed] (p_4) -- (p_1);
\path[edge, dashed] (n_4) -- (n_1);

\end{scope}

\begin{scope}[shift={(8,0)}]
\foreach \x/\y/\z in {-6/1/b,-2/1/c,2/1/d,6/1/b1}{
\node[vertex] (\z) at (\x,\y){};
}

\foreach \x/\y/\z in {-6/0.5/a,-2/0.5/b,2/0.5/u,6/0.5/a}{
\node[ver] () at (\x,\y){$\z$};
}

\foreach \x/\y/\z in {-4/3/p_2,0/3/p_4,4/3/p_1,8/3/p2}{
\node[vertex] (\z) at (\x,\y){};
}

\foreach \x/\y/\z in {-4/3.5/p_4,0/3.5/p_3,4/3.5/p_1,8/3.5/p_4}{
\node[ver] () at (\x,\y){$\z$};
}

\foreach \x/\y/\z in {-4/-1/n_2,0/-1/n_4,4/-1/n_1,8/-1/q2}{
\node[vertex] (\z) at (\x,\y){};
}

\foreach \x/\y/\z in {-4/-1.5/n_4,0/-1.5/n_3,4/-1.5/n_1,8/-1.5/n_4}{
\node[ver] () at (\x,\y){$\z$};
}

\foreach \x/\y in {b/c,c/d,d/b1,b/p_2,c/p_2,b/n_2,c/n_2,c/p_4,d/p_4,c/n_4,d/n_4,d/p_1,b1/p_1,d/n_1,b1/n_1,
b1/p2,b1/q2}{
\path[edge] (\x) -- (\y);}

\foreach \x/\y in {p_2/p_4,p_1/p2,n_2/n_4,n_1/q2}{
\path[edge, dotted] (\x) -- (\y);}
\path[edge, dashed] (p_4) -- (p_1);
\path[edge, dashed] (n_4) -- (n_1);

\end{scope}

\node[ver] () at (-8.3,-3){$\lk(u,\Delta)$};
\node[ver] () at (10,-3){$\lk(v,\Delta)$};

\end{tikzpicture}
\caption{$\lk(u,\Delta)$ and $\lk(v,\Delta)$ in $\Delta$.}\label{fig:lk(u&v)}
\end{figure}

\begin{proof}
Let $\lk(auv,\Delta)=\{p_1,n_1\}$, $\lk(abu,\Delta)=\{p_2,n_2\}$, $\lk(buv,\Delta)=\{p_3,n_3\}$ and $\lk(abv,\Delta)$ $=\{p_4,n_4\}$ in $\Delta$. The closed curve $\partial(buv)$ in $\lk(a,\Delta)$ separates $\lk(a,\Delta)$ into two parts which we denote by $S^{+}_a$ and $S^{-}_a$. Without loss of generality, we can assume that $p_1,p_2,p_4\in S^{+}_a$ and $n_1,n_2,n_4\in S^{-}_a$. The closed curve $\partial(auv)$  in $\lk(b,\Delta)$ separates $\lk(b,\Delta)$ into two parts $S^{+}_b$ and $S^{-}_b$. Then there is a path $P(u,p_2,...,p_4,v) \in \lk(b,\lk(a,\Delta))$, which also lies in $\lk(a,\lk(b,\Delta))$, and hence  either in $S^{+}_b$ on in $S^{-}_b$. Without loss of generality, we can assume that  the path $P(u,p_2,...,p_4,v)$ lies in $S^{+}_b$, and hence $p_2,p_4 \in S^{+}_b$ and $n_2,n_4\in S^{-}_b$. Similarly we can assume that $p_3 \in S^{+}_b$ and $n_3 \in S^{-}_b$. See Figure \ref{fig:lk(a&b)} for $\lk(a,\Delta)$ and $\lk(b,\Delta).$ A dotted line between $p_i$ and $p_j$ represents a path between those two vertices. Here, $p_i$ and $p_j$ (resp., $n_i$ and $n_j$) may be equal but $p_i \neq n_j$ for all $i, j\in \{1,2,4\}$ or $\{2,3,4\}$. Further, the paths $P(p_i,...,p_j)$ and $P(p_k,...,p_l)$ (resp., $P(n_i,...,n_j)$ and $P(n_k,...,n_l)$) may intersect but $P(p_i,...,p_j)$ and $P(n_k,...,n_l)$ do not intersect for  $i,j,k,l \in \{1,2,4\}$ or $\{2,3,4\}$. 

Now there are paths $P(b,p_2,...,p_1,v)$, $P(b,n_2,...,n_1,v)$  in $\lk(u,\lk(a,\Delta))= \lk(a,\lk(u,\Delta))$ and   paths $P(a,p_2,...,p_3,v)$, $P(a,n_2,...,n_3,v)$  in $\lk(u,\lk(b,\Delta))=\lk(b,\lk(u,\Delta))$. Thus, a portion  of $\lk(u,\Delta),$ including the dotted lines, is as in Figure \ref{fig:lk(u&v)}. Similarly, there are paths $P(b,p_4,...,p_1,u)$, $P(b,n_4,...,n_1,u)$  in $\lk(v,\lk(a,\Delta))= \lk(a,\lk(v,\Delta)),$ and paths $P(a,p_4,...,p_3,u)$, $P(a,n_4,...,n_3,u)$  in $\lk(v,\lk(b,\Delta))=\lk(b,\lk(v,\Delta))$. Thus,  a portion  of $\lk(v,\Delta),$ including the dotted lines, is also as in Figure \ref{fig:lk(u&v)}.

Since $\lk(v,\lk(u,\Delta))$ is a circle, either there are paths $P(a,p_1,...,p_3,b)$ and $P(a,n_1,...,n_3,b)$ in $\lk(v,\lk(u,\Delta))=\lk(u,\lk(v,\Delta))$, or there are paths $P(a,p_1,...,n_3,b)$ and $P(a,n_1,...,p_3,b)$ in $\lk(v,\lk(u,\Delta))=\lk(u\lk(v,\Delta))$.  In the first case, small neighborhoods of  $|\partial(abv)|$ in $|\lk(u,\Delta)|$ and    $|\partial(abu)|$ in $|\lk(v,\Delta)|$ are annuli.  In the second case small neighborhoods of the same   missing triangles are M\"obius strips.


\end{proof}

\begin{lemma}\label{lemma:missingtetra2}
Let $\Delta$ be a 3-dimensional normal pseudomanifold. Let $\tau=abcd$ be a missing facet in $\Delta$ such that $(i)$ for $x\in\{b,c,d\}$, $\partial(\tau \setminus \{x\})$ separates  $\lk(x,\Delta),$ and $(ii)$  $\partial(bcd)$  does not separate  $\lk(a,\Delta)$. Then there exists $\Delta'$ a three-dimensional normal pseudomanifold such that  $\Delta = (\Delta')^\psi_a$ is obtained from a vertex folding at $a \in \Delta'$  and $abcd$ is the image of the removed facet.  

\end{lemma}

\begin{proof}
Since $bcd \in \Delta$ its link consists of two vertices which we designate $p$ and $n$. For $x\in\{b,c,d\}$, $\partial(\tau \setminus\{x\})$ separates  the link of $x$  into two parts $S^{+}_x$ and $S^{-}_x$, where $S^{+}_x$ and $S^{-}_x$ are connected surfaces with common boundary $\partial(\tau\setminus\{x\}), ~p \in S^+_x$ and $n \in S^-_x.$   Let $V^{+}_x$  be the collection of the vertices other than $a,b,c,d$ of $S^{+}_x.$  Similarly, let $V^{-}_x$ be the collection of the vertices other than $a,b,c,d$ of $S^{-}_x.$ Let $E^{\pm}_x=\{wx: w\in V^{\pm}_x\}$. Finally, let $E^{\pm}=E^{\pm}_b\cup E^{\pm}_c\cup E^{\pm}_d$.  By definition $E^+ \cap E^- = \emptyset.$

Suppose  $\sigma \not \subseteq bcd$ is a face of $\Delta$  which does not contain $a$ and intersects $bcd$. Let $E_\sigma$ be the edges of $\sigma$ which have exactly one end point in  $bcd$.  So $E_\sigma \subseteq E^+ \cup E^-.$  In order to construct $\Delta'$ we want to show that $E_\sigma \subset E^+$ or $E_\sigma \subset E^-.$ 
If a face $\sigma$ intersects $bcd$ then, for some $x\in\{b,c,d\}$, $\sigma \in \Star  (x,\Delta)$. Therefore $V(\sigma)\setminus\{b,c,d\}\subset V^{+}_x$ or $V^{-}_x$.   If $V(\sigma)\setminus\{b,c,d\}\subset V^{\pm}_x$, then $V(\sigma)\setminus\{b,c,d\}  \cap V^{\mp}_y = \emptyset$ for all $y \in \{b,c,d\}.$  To see this, consider the link of $xy.$ It is a circle, contains the path $P(p,a,n)$ and the remaining vertex $z$ of $\{b,c,d\} \setminus \{x,y\}.$  So all of the vertices of $\sigma \setminus \{a,b,c,d\}$ all lie in exactly one of the  semicircles with endpoints $a$ and $z,$ one of which contains $p$ and the other $n.$  Hence $E_\sigma$ is a subset of exactly one of  $E^+$ or  $E^-$.

   Let $\tilde\Delta$ be the induced sub-complex of $\Delta$ restricted to  the vertex set $V(\Delta)\setminus \{a\}$. Then $\tilde\Delta$ is a 3-dimensional normal pseudomanifold with boundary $\lk(a,\Delta)$. On our way to finding $\Delta'$ we construct an intermediate complex $\tilde\Delta'$ which is closely related to $\tilde\Delta$ and whose vertex set is $V(\tilde\Delta) \cup \{b',c',d'\}$ where $\{b',c',d'\}$ are three new vertices. Let  $\tilde\sigma$ be a face of $\tilde\Delta$ not contained in $bcd.$  Define $\tilde\sigma'$ to be $\tilde\sigma$ if $\tilde\sigma$ does not intersect $bcd$ or $E_{\tilde\sigma} \subset E^+.$  If $E_{\tilde\sigma} \subset E^-,$ then $\tilde\sigma'$ is defined to be $\tilde\sigma$ with all occurrences of $b,c,$ and $d$ replaced with $b',c',$ and $d'$ respectively.    Now set $\tilde\Delta'$ to be the complex whose facets are $\{\tilde\sigma': \tilde\sigma \mbox{ a facet of } \tilde\Delta\} \cup \{nb'c'd'\}.$
   
   The codimension-one face $bcd$ of $\tilde\Delta'$ is contained in one facet, $pbcd$.  Similarly, the codimension-one face $b'c'd'$ is contained in one facet, $nb'c'd'.$  All other codimension-one faces $\tilde\sigma'$ of $\tilde\Delta'$ are contained in one or two facets depending on whether $\tilde\sigma$ was contained in one or two facets in $\tilde\Delta.$ 
   
   By Lemma \ref{lemma:missingtetra1}  a small neighborhood of $|\partial(bcd)|$ in $|\lk(a,\Delta)|$ is an annulus. Since $\partial (bcd)$ does not separate $\lk (a,\Delta),$ this link was formed via handle addition \cite[Lemma 3.3]{BagchiDatta}.  So write $\lk(a,\Delta) =  S^\psi$, where is $S$ is a connected triangulated compact surface and the image of the domain of $\psi$ is $\{b,c,d\}.$  The triangles of $\tilde\Delta'$ contained in exactly one facet are constructed in exactly the same fashion as $S.$  Hence $\tilde\Delta'$ is a pseudomanifold with boundary $S.$  Finally, set $\Delta'$ to be $\tilde\Delta'$ with its boundary coned off by $a.$  So $\Delta'$ is a pseudomanifold and $\lk(a,\Delta')$ is $S.$  Furthermore, from the construction, $\Delta = (\Delta')^\psi,$ where $\psi:\{a,b',c',d'\} \to \{a,b,c,d\}.$  Two vertices $x,y$ in $\Delta'$ can be connected by a path by concatenating the lift of a path in $\Delta$ from $x$ to $a$ and from $a$ to $y.$ Thus $\Delta'$ is connected.  To finish the proof we show that the links of all other vertices are connected surfaces.   
  
  Let $x \in \{b,c,d\}.$   Since $\lk(x,\Delta)= \lk(x,\Delta') \#_\psi \lk(x',\Delta')$ and the link of $x$ in $\Delta$ is a connected surface, each of the links $\lk (x,\Delta')$ and $\lk(x', \Delta')$ must be connected surfaces.  Lastly, suppose $v$ is a vertex of $\Delta'$ other than $a,b,c,d,b',c'$ or $d'.$  The link of $v$ in $\Delta'$ is the same as the link of $v$ in $\Delta,$ except that for each of $x \in \{b,c,d\}$ with $vx \in E^-,$ $x$ is replaced by $x'$ in the link of $v$ in $\Delta'$. Thus the link of $v$ in $\Delta'$ is a connected surface and $\Delta'$ is a normal pseudomanifold.  
\end{proof}

Similar principles allow us to identify edge foldings via a missing tetrahedron and the topology of the induced missing triangles in the links of its vertices.  
 
\begin{lemma}\label{lemma:missingtetra3}
Let $\Delta$ be a 3-dimensional normal pseudomanifold. Let $\tau=abuv$ be a missing facet in $\Delta$ such that $(i)$ for $x\in\{a,b\}$, $\partial(\tau \setminus \{x\})$ separates  $\lk(x,\Delta),$ and $(ii)$ a small neighborhood of $|\partial(abv)|$ in $|\lk(u,\Delta)|$ is a M\"{o}bius strip. Then there exists $\Delta'$ a three-dimensional normal pseudomanifold such that  $\Delta = (\Delta')^\psi_{uv}$ is obtained from an edge folding at $uv \in \Delta'$  and $abuv$ is the removed facet.
\end{lemma}

\begin{proof}
Let $\lk(auv,\Delta)=\{p_1,n_1\}$, $\lk(abu,\Delta)=\{p_2,n_2\}$, $\lk(buv,\Delta)=\{p_3,n_3\}$ and $\lk(abv,\Delta)$ $=\{p_4,n_4\}$ in $\Delta$. For $x\in\{a,b\}$, the closed curve $\partial(\tau \setminus \{x\})$ separates $\lk(x,\Delta)$ into two parts which we denote by $S^{+}_x$ and $S^{-}_x$. Then we choose $p_i$, $n_i$ as in the proof of Lemma \ref{lemma:missingtetra1} (cf. Figure \ref{fig:lk(a&b)}). In particular, for $1 \le i \le 4$ and $x \in \{a,b\},~p_i \in S^+_x$ and $n_i \in S^-_x.$ Let $V^{+}_x$  be the collection of the vertices other than $a,b,u,v$ in $S^{+}_x.$    Similarly, let $V^{-}_x$ be the collection of the vertices other than $a,b,u,v$ in $S^{-}_x$. 

Now, we define a new complex $\Delta'$ such that $V(\Delta')=(V(\Delta')\setminus\{a,b\})\sqcup\{a^+,a^-,b^+,b^-\}$. The facets of $\Delta'$ is the set $\{\sigma': \sigma \mbox{ a facet of } \Delta\}\cup \{a^+b^+uv,a^-b^-uv\}$, where $\sigma'$ is obtained from $\sigma$ as follows. $(i)$ If $a,b\not \in \sigma$ then $\sigma'=\sigma$. $(ii)$ If $\sigma \cap abuv = a$ then the other  vertices of $\sigma$ lie either in $V^{+}_a$ or in $V^{-}_a$. Therefore, $\sigma'=\sigma\setminus \{a\} \cup \{a^+\}$ (resp., $=\sigma\setminus \{a\} \cup \{a^-\}$) if $V(\sigma\setminus \{a\}) \subset V^{+}_a$ (resp., $V(\sigma\setminus \{a\}) \subset V^{-}_a$). The case $\sigma \cap abuv = b$ is similar to this. $(iii)$ If $\sigma \cap abuv = au$ then the other two vertices say $w,z$ lie either in $V^{+}_a$ or in $V^{-}_a$ as $uwz$ forms a triangle in the link of $a$. Therefore, $\sigma'=\sigma\setminus \{a\} \cup \{a^+\}$ (resp., $=\sigma\setminus \{a\} \cup \{a^-\}$) if $V(\sigma\setminus \{a,u\}) \subset V^{+}_a$ (resp., $V(\sigma\setminus \{a,u\}) \subset V^{-}_a$). The cases $\sigma \cap abuv = av, bu, bv$ are treated similarly. $(iv)$ If $\sigma \cap abuv = ab$ then the other two vertices say $w,z$ lie either in $V^{+}_a$ or in $V^{-}_a$ as $bwz$ forms a triangle in the link of $a$. If $w,z$ lie in $V^{+}_a$ then $w,z \in P(p_2,\dots,p_4)$ and hence by  construction $w,z$ lie in $V^{+}_b$. Therefore $\sigma'=\sigma\setminus \{ab\} \cup \{a^+b^+\}$. The subcase when $w,z$ lie in $V^{-}_a$ is similar to this. $(v)$ when $\sigma \cap abuv \in \{abu,abv,auv,buv\}$ the remaining vertex $w$ of $\sigma$ is $p_i$ or $n_i$ for some $i.$  When $w$ is $p_i,~\sigma'$ is $\sigma$ with $a^+$ replacing $a$ and $b^+$ replacing $b.$  When $w$ is $n_i$ we use $a^-$ and $b^-$ as replacements for $a$ and $b$ instead.

The bijection $\psi: a^+b^+uv \to a^-b^-uv$ which sends $(a^+,b^+,u,v) \to (a^-,b^-,u,v)$ is  folding admissible at the edge $uv$. Further,  $\Delta$ is obtained from $\Delta'$  by identifying all faces $\rho_1 \subseteq a^+b^+uv $ and $\rho_2 \subseteq a^+b^+uv$ such that $\psi(\rho_1) = \rho_2,$ and then removing the facet formed by identifying $a^+b^+uv$ and $a^-b^-uv$, where $a$ (resp., $b$) is the vertex in $\Delta$ by identifying the vertices $a^+$ and $a^-$ (resp., $b^+$ and $b^-$) in $\Delta'$. So $\Delta = \Delta^\psi_{uv}.$

Now, $\lk(a,\Delta)= \lk(a^+,\Delta') \#_{buv} \lk(a^-,\Delta')$ and $\lk(b,\Delta)= \lk(b^+,\Delta') \#_{auv} \lk(b^-,\Delta')$. Since the links of $a$ and $b$ in $\Delta$ are connected surfaces, each of the links $\lk (x,\Delta')$ for $x=a^+,a^-,b^+,b^-$ must be connected surfaces.  Suppose $y$ is a vertex of $\Delta'$ other than $a^+,b^+,a^-,b^-,u$ or $v$. Then the link of $y$ in $\Delta'$ is the same as the link of $y$ in $\Delta,$ except that for each of $x \in \{a,b\}$ with $y \in S_x^+$ (resp,  $y \in S_x^-$) in $\Delta$, $x$ is replaced by $x^+$ (resp,  $x^-$) in the link of $y$ in $\Delta'$. Now, we consider the link of $u$ and $v$ in $\Delta'$. Since a small neighborhood of $|\partial(abv)|$ in $|\lk(u,\Delta)|$ is a M\"{o}bius strip,  a portion of $\lk(u,\Delta)$ is as in Figure \ref{fig:lk(u)}.

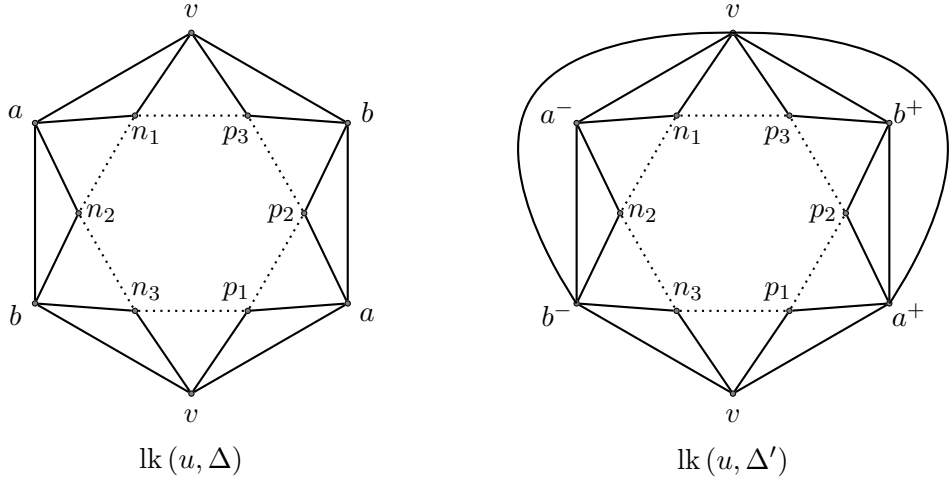
\begin{figure}[ht]
\tikzstyle{ver}=[]
\tikzstyle{vertex}=[circle, draw, fill=black!50, inner sep=0pt, minimum width=2pt]
\tikzstyle{edge} = [draw,thick,-]
\centering
\begin{tikzpicture}[scale=0.6]
\begin{scope}[shift={(-6,0)}]
\foreach \x/\y in {90/v',150/a',210/b'}{
    \node[vertex] (\y) at (\x:4){};
}

\foreach \x/\y in {90/v,150/a,210/b}{
\node[ver] () at (\x:4.5){$\y$};
}

\foreach \x/\y in {270/v,330/a,30/b}{
\node[ver] () at (\x:4.5){$\y$};
    \node[vertex] (\y) at (\x:4){};
}

\foreach \x/\y in {v/a,a/b,b/v',v'/a',a'/b',b'/v}{
\path[edge] (\x) -- (\y);}

\foreach \x/\y in {300/p_1,0/p_2,60/p_3,120/n_1,180/n_2,240/n_3}{
\node[ver] () at (\x:2){$\y$};
    \node[vertex] (\y) at (\x:2.5){};
}

\foreach \x/\y in {v/p_1,a/p_1,a/p_2,b/p_2,b/p_3,v'/p_3,v'/n_1,a'/n_1,a'/n_2,b'/n_2,b'/n_3,v/n_3}{
\path[edge] (\x) -- (\y);}

\foreach \x/\y in {n_1/n_2,n_2/n_3,n_3/p_1,p_1/p_2,p_2/p_3,p_3/n_1}{
\path[edge, dotted] (\x) -- (\y);}
\end{scope}

\begin{scope}[shift={(6,0)}]
\foreach \x/\y in {90/v',150/a',210/b'}{
    \node[vertex] (\y) at (\x:4){};
}

\foreach \x/\y in {90/v,150/a^-,210/b^-}{
\node[ver] () at (\x:4.5){$\y$};
}

\foreach \x/\y in {270/v,330/a,30/b}{
    \node[vertex] (\y) at (\x:4){};
}

\foreach \x/\y in {270/v,330/a^+,30/b^+}{
\node[ver] () at (\x:4.5){$\y$};
}

\foreach \x/\y in {v/a,a/b,b/v',v'/a',a'/b',b'/v}{
\path[edge] (\x) -- (\y);}

\foreach \x/\y in {300/p_1,0/p_2,60/p_3,120/n_1,180/n_2,240/n_3}{
\node[ver] () at (\x:2){$\y$};
    \node[vertex] (\y) at (\x:2.5){};
}

\foreach \x/\y in {v/p_1,a/p_1,a/p_2,b/p_2,b/p_3,v'/p_3,v'/n_1,a'/n_1,a'/n_2,b'/n_2,b'/n_3,v/n_3}{
\path[edge] (\x) -- (\y);}

\foreach \x/\y in {n_1/n_2,n_2/n_3,n_3/p_1,p_1/p_2,p_2/p_3,p_3/n_1}{
\path[edge, dotted] (\x) -- (\y);}

\draw[edge] plot [smooth,tension=1] coordinates{(v')(-4.5,2.5)(b')};
\draw[edge] plot [smooth,tension=1] coordinates{(v')(4.5,2.5)(a)};
\end{scope}

\node[ver] () at (-6,-5.5){$\lk(u,\Delta)$};
\node[ver] () at (6,-5.5){$\lk(u,\Delta')$};
\end{tikzpicture}
\caption{$\lk(u,\Delta)$ and $\lk(u,\Delta')$ in $\Delta$ and $\Delta'$ respectively.} \label{fig:lk(u)}
\end{figure}

Here the paths $P(n_1,\dots,n_2)$ and $P(p_1,\dots,p_2)$ do not intersect but the paths $P(n_1,\dots,n_2)$ and $P(p_2,\dots,p_3)$ may intersect. Similarly, the paths $P(n_2,\dots,n_3)$ and $P(p_2,\dots,p_3)$ do not intersect but the paths $P(n_2,\dots,n_3)$ and $P(p_1,\dots,p_2)$ may intersect. By Lemma \ref{lemma:missingtetra1}, a small neighborhood of $|\partial(abu)|$ in $|\lk(v,\Delta)|$ is also a M\"{o}bius strip. Thus,  a portion of $\lk(v,\Delta)$ is as in Figure \ref{fig:lk(v)}.

\begin{figure}[ht]
\tikzstyle{ver}=[]
\tikzstyle{vertex}=[circle, draw, fill=black!50, inner sep=0pt, minimum width=2pt]
\tikzstyle{edge} = [draw,thick,-]
\centering
\begin{tikzpicture}[scale=0.6]
\begin{scope}[shift={(-6,0)}]
\foreach \x/\y in {90/u',150/a',210/b'}{
    \node[vertex] (\y) at (\x:4){};
}

\foreach \x/\y in {90/u,150/a,210/b}{
\node[ver] () at (\x:4.5){$\y$};
}

\foreach \x/\y in {270/u,330/a,30/b}{
\node[ver] () at (\x:4.5){$\y$};
    \node[vertex] (\y) at (\x:4){};
}

\foreach \x/\y in {u/a,a/b,b/u',u'/a',a'/b',b'/u}{
\path[edge] (\x) -- (\y);}

\foreach \x/\y in {300/p_1,0/p_4,60/p_3,120/n_1,180/n_4,240/n_3}{
\node[ver] () at (\x:2){$\y$};
    \node[vertex] (\y) at (\x:2.5){};
}

\foreach \x/\y in {u/p_1,a/p_1,a/p_4,b/p_4,b/p_3,u'/p_3,u'/n_1,a'/n_1,a'/n_4,b'/n_4,b'/n_3,u/n_3}{
\path[edge] (\x) -- (\y);}

\foreach \x/\y in {n_1/n_4,n_4/n_3,n_3/p_1,p_1/p_4,p_4/p_3,p_3/n_1}{
\path[edge, dotted] (\x) -- (\y);}
\end{scope}

\begin{scope}[shift={(6,0)}]
\foreach \x/\y in {90/u',150/a',210/b'}{
    \node[vertex] (\y) at (\x:4){};
}

\foreach \x/\y in {90/u,150/a^-,210/b^-}{
\node[ver] () at (\x:4.5){$\y$};
}

\foreach \x/\y in {270/u,330/a,30/b}{
    \node[vertex] (\y) at (\x:4){};
}

\foreach \x/\y in {270/u,330/a^+,30/b^+}{
\node[ver] () at (\x:4.5){$\y$};
}

\foreach \x/\y in {u/a,a/b,b/u',u'/a',a'/b',b'/u}{
\path[edge] (\x) -- (\y);}

\foreach \x/\y in {300/p_1,0/p_4,60/p_3,120/n_1,180/n_4,240/n_3}{
\node[ver] () at (\x:2){$\y$};
    \node[vertex] (\y) at (\x:2.5){};
}

\foreach \x/\y in {u/p_1,a/p_1,a/p_4,b/p_4,b/p_3,u'/p_3,u'/n_1,a'/n_1,a'/n_4,b'/n_4,b'/n_3,u/n_3}{
\path[edge] (\x) -- (\y);}

\foreach \x/\y in {n_1/n_4,n_4/n_3,n_3/p_1,p_1/p_4,p_4/p_3,p_3/n_1}{
\path[edge, dotted] (\x) -- (\y);}

\draw[edge] plot [smooth,tension=1] coordinates{(u')(-4.5,2.5)(b')};
\draw[edge] plot [smooth,tension=1] coordinates{(u')(4.5,2.5)(a)};
\end{scope}

\node[ver] () at (-6,-5.5){$\lk(v,\Delta)$};
\node[ver] () at (6,-5.5){$\lk(v,\Delta')$};
\end{tikzpicture}
\caption{$\lk(v,\Delta)$ and $\lk(v,\Delta')$ in $\Delta$ and $\Delta'$ respectively.} \label{fig:lk(v)}
\end{figure}
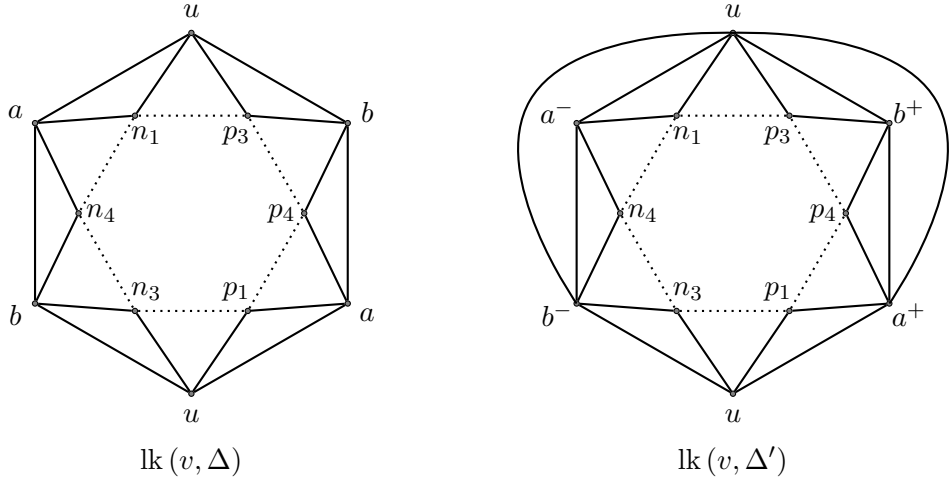

Now,  the vertices $a,b$ (in the link of $u$ and $v$ in $\Delta$) connected with $p_i$ will be replaced by $a^+,b^+$ in the link of $u$ and $v$ in $\Delta'$ and the vertices $a,b$ (in the link of $u$ and $v$ in $\Delta$) connected with $n_i$ in the will be replaced by $a^-,b^-$ in the link of $u$ and $v$ in $\Delta'$ (see Figures \ref{fig:lk(u)} and \ref{fig:lk(v)}). Since $a^+b^+uv$ and $a^-b^-uv$ are new facets in $\Delta'$, links of $u$ and $v$ in $\Delta'$ are closed connected surface with  $b_1(\lk(u,\Delta'))=b_1(\lk(u,\Delta))-1$ and $b_1(\lk(v,\Delta'))=b_1(\lk(v,\Delta))-1$. Thus all the vertex-links in $\Delta'$ are connected surfaces. It is easy to see that $\Delta'$ is connected and  hence it is a normal pseudomanifold.
\end{proof}

\noindent   If $\Delta = \Delta'^\psi_{ef},$ then we call $\Delta'$ an {\bf edge unfolding} of $\Delta.$

What happens if in the above lemma if we assume that a neighborhood of $|\partial abv|$ in the link of $u$ is an annulus?  The proof of the lemma is unchanged until we reach Figures \ref{fig:lk(u)} and \ref{fig:lk(v)}.  Instead the link of $u$ in $\Delta$   can be seen as a   triangulation of a surface with two identified boundary components labeled by $C(a,b,v).$  This leads to a diagram depicting the link of $u$ in $\Delta'$ as being a  triangulation of a surface with two boundary components labeled by $C(a^+,b^+,v)$ and $C(a^-,b^-,v),$ and then filling in these two circles with triangles $a^+b^+v, a^-b^-v.$  A similar description holds for the link of $v$ in $\Delta'.$  Thus the link of $u$ (and $v$) is a  triangulation of a  surface with a pair of vertices identified.  In this case the unfolding of $\Delta$ is not a normal pseudomanifold.  However, as we will see in Lemma \ref{no outside vertices}, the unfolding of $\Delta$ will have enough similarities to a normal pseudomanifold for our purposes.

\section{Rigidity} \label{rigidity}

In \cite{Kalai} Kalai introduced the idea of using infinitesimal generic rigidity of graphs as a tool for studying face enumeration.  Since this innovative and influential paper rigidity has become part of the tool box available for analyzing $f$-vectors, especially  $g_2.$  See, for instance, \cite{NevoNovinsky},\cite{Zheng}. Here we present the notation and results we need from this subject.   

For now $G$ is a fixed simple graph (no loops or multiple edges) with vertex set $V$ and edge set $E.$    
Let $f:V \to \RR^d.$    We say  $(G,f)$ is {\bf rigid} if there exists $\varepsilon>0$ such that if $g:V \to \RR^d$ with $||f(v)-g(v)||< \varepsilon$ for all $v \in V$ and $||f(v)-f(u)||=||g(v)-g(u)||$ for all $uv \in E,$ then $||f(u)-f(v)||=||g(u)-g(v)||$ for all $u,v \in V.$  In other words, maps close $f$ which preserve distances between edges of $G$ preserve  distances between all vertices.  

\begin{definition}\cite{Kalai}
If for generic choices of $f: V \to \RR^d,~(G,f)$ is rigid, then $G$ is called  {\bf generically $d$-rigid.}  A simplicial complex $\Delta$ is $d$-rigid if $G(\Delta)$ is generically $d$-rigid.
\end{definition}
\noindent   Here generic means that there is a (nonempty) Zariski open subset of $(\RR^d)^V$ all of whose corresponding  $f$'s  satisfy the definition.

Note that it is immediate from the definition that if $G$ is generically $d$-rigid and $e$ is an edge whose vertices are in $G,$ but $e \notin G,$ then $G \cup \{e\}$ is generically $d$-rigid.  The following lemma is at the core of the next theorem.

\begin{lemma} \label{union lemma} (Union lemma) \cite{Kalai}
If $G_1$ and $G_2$ are generically $d$-rigid and $G_1 \cap G_2$ contains a $K_d$, then $G_1 \cup G_2$ is $d$-rigid.  Consequently, if $\Delta_1$ and $\Delta_2$ are $d$-rigid and $\Delta_1 \cap \Delta_2$ contains a $(d-1)$-simplex, then $\Delta_1 \cup \Delta_2$ is  $d$-rigid. 
\end{lemma}

\begin{theorem}  \cite{Fogelsanger},\cite{Kalai}
Normal $d$-pseudomanifolds are  $(d+1)$-rigid if $d \ge 2.$
\end{theorem}

 One of the main reasons rigidity has been a useful notion in the study of edge enumeration is the use of stresses.  Let $\Delta$ be  $d$-rigid.
\begin{definition}  \cite{Kalai}
Let $f:V(\Delta) \to \RR^d$ be a function.  A {\bf stress} of $f$ is a function $\omega:E(\Delta) \to \RR$ such that for every vertex $v \in V,$
$$\displaystyle\sum_{vu \in E} \omega(vu) (f(v) -f(u)) = 0.$$
The set of all stresses of $f$ is an $\RR$-vector space which we denote by $\mathcal{S}(\Delta_f).$
\end{definition}

\begin{proposition} \label{stresses=g2}  \cite{Kalai}
If $\dim \Delta=d,$ $d \ge 2, \Delta$ is  $(d+1)$-rigid and $f$ is a generic map from $V$ to $\RR^{d+1},$  then $\dim \mathcal{S}(\Delta_f) = g_2(\Delta).$ 
\end{proposition}

One of the advantages of using the stress space to estimate $g_2$ is that it is frequently possible to estimate it from local conditions instead of the  global definition of $g_2.$  The {\bf support} of a stress $\omega$ is the set of edges $uv$ such that $\omega(uv) \neq 0.$ 

\begin{lemma} \label{cone lemma} (Cone lemma) \cite{TayWhiteWhiteley},\cite{NevoNovinsky}
Suppose $\Delta$ is  $d$-rigid.  Let $C(\Delta)$ be the cone of $\Delta$ with cone vertex $v.$  Then $C(\Delta)$ is $(d+1)$-rigid.  Furthermore, if $g_2>0,$ then for generic $f:V(C(\Delta)) \to \RR^{d+1}$ there is a stress in $\mathcal{S}(C(\Delta)_f)$ whose support  contains an edge incident to $v.$ 
\end{lemma}

A stress $\omega$ on a subcomplex of $\Delta$ can be extended to $\Delta$ by simply setting $\omega(uv)=0$ for any $uv$ not in the subcomplex.  Since $g_2(C(\Delta)) =g_2(\Delta),$ Proposition \ref{stresses=g2} and the cone lemma imply the following lower bounds  for $g_2$  originally due to Kalai.

\begin{lemma} \label{g2 subcomplex < g2 complex}
  Suppose $\Delta$ is $d$-dimensional, $(d+1)$-rigid and $\Omega$ is a $(d+1)$-rigid subcomplex of $\Delta.$ Then $g_2(\Omega) \le g_2(\Delta).$  
\end{lemma}

\begin{lemma} \label{g2>face} \cite[Theorem 7.3]{Kalai}
If $\Delta$ is normal $d$-pseudomanifold with $d \ge 3$ and $\sigma$ is a face of $\Delta$ of codimension three or more, then $g_2(\Delta) \ge g_2(\Star \sigma)=g_2(\lk \sigma).$  
\end{lemma}

\section{Relatively minimal $g_2$}

\noindent The previous lemma naturally leads to the definition of complexes with relatively minimal $g_2.$

\begin{definition}
A simplicial complex $\Delta$ has {\bf relatively minimal $g_2$} if 
  \begin{itemize}
    \item $\Delta$ is a normal pseudomanifold of dimension at least three.
    \item There exists a nonempty face $\sigma$ of codimension at least three such that $g_2(\Delta) = g_2(\lk (\sigma)).$
  \end{itemize} 
\end{definition}

\noindent If we wish to specify $\sigma,$ then we say $\Delta$ has relatively minimal $g_2$ with respect to $\sigma.$ 

 Applications of the relevant formulas from Section \ref{constructions} give several ways to produce complexes with relatively minimal $g_2.$  Here are some examples.  

\begin{proposition} \label{relatively minimal preserving operations} $\mbox{ }$

\begin{enumerate}  

\item  The boundary of a simplex $\sigma, \dim \sigma \ge 4$  has relatively minimal $g_2$ with respect to any face of codimension three or more.
\item  \label{facet subdivision} Suppose $\Delta$ is a normal psuedomanifold of dimension at least three. If $\Delta'$ is obtained from $\Delta$ by subdividing one (or more) facets, then $\Delta'$ has relatively minimal $g_2$ if and only if $\Delta$ has relatively minimal $g_2.$  Since stacked spheres are formed by starting with the boundary of a simplex and then subdividing facets, all stacked spheres ($\dim \ge3$) have relatively minimal $g_2.$

\item \label{relatively minimal via connected sum} Suppose $\Delta = \Delta_1 \#_\psi \Delta_2,$ where $\Delta_1$ and $\Delta_2$ are normal $d$-pseudomanifolds with $d \ge 3.$ If $v$ is not one of the identified vertices, then $v$ is in $\Delta_1$ or $\Delta_2.$ Say $v$ is in  $\Delta_1$.  Then  $\Delta$ has relatively minimal $g_2$ with respect to $v$ if and only if $\Delta_1$ has relatively minimal $g_2$ with respect to $v$ and $\Delta_2$ is a stacked sphere.  Alternatively, assume $v$ is the vertex representing $v_1 \in \Delta_1$ and $\psi(v_1)$ in $\Delta_2.$  Then $\Delta$ has relatively minimal $g_2$ with respect to $v$ if and only if $\Delta_1$ has relatively minimal $g_2$ with respect to $v_1$ and $\Delta_2$ has relatively minimal $g_2$ with respect to $\psi(v_1).$ 
 
\item  Suppose $\Delta$ is a normal psuedomanifold of dimension at least three.  Let $\psi$ be an admissible folding at $x.$  Then $\Delta^\psi_x$ has relatively minimal $g_2$ with respect to $\bar{x}$ if and only if $\Delta$ has relatively minimal $g_2$ with respect to $x.$ 

\item  Suppose $\Delta$ is a normal pseudomanifold of dimension at least two and $v$ is a graph cone point of $\Delta.$  Then $\Sigma_v \Delta$ has relatively minimal $g_2$ with respect to either suspension point. 

\item  Suppose $\Delta$ is a normal pseudomanfiold of dimension at least three and $\psi$ is an edge admissible folding at $uv$ such that $\Delta^\psi_{uv}$ is a normal pseudomanifold.  Then $\Delta^\psi_{uv}$ has relatively minimal $g_2$ with respect to  $\bar{v}$ if and only if $\Delta$ has relatively minimal $g_2$ with respect to  $v.$

\end{enumerate}
\end{proposition}

\noindent  As we will see in Theorem \ref{no outside vertices}, connected sum with a stacked sphere (including the boundary of $4$-simplex)  is a special case of (\ref{facet subdivision}).  In Theorem \ref{two singularities - combinatorics} we will want to refer to these two operations separately.

The above operations already provide a class of three-dimensional complexes with relatively minimal $g_2.$  Recall that, a   pseudocompression body is a connected compression body with each boundary component coned off.  As we will see in moment, all pseudocompression bodies have relatively minimal $g_2$ triangulations.  One of the main results of Section \ref{few singularities} is that all triangulations with relatively minimal $g_2$ and two or fewer singularities  are  triangulations of pseudocompression bodies.

\begin{proposition}
  If $X$ is a pseudocompression body, then there exists triangulations of $X$ which have relatively minimal $g_2$  with respect to $v,$ where $v$ is the cone point for the top part of the boundary of its associated compression body.  
\end{proposition}

 \begin{proof}
 Write $X$ as a connected compression body $C$ with its boundary components coned off.  So $C$ is $K \times [0,1],$ where $K$ is a compact surface, with $1$-handles $H_1, \dots, H_m$ attached to the top of $C.$  We proceed by induction on $m,$ the number of attached $1$-handles.  If $m=0,$ then since $X$ is connected, $K$ is connected and $X$ is the suspension of $K.$  Let $\Delta$ be a triangulation of $K$ with graph cone point $v.$  For $K=S^3,~\Delta$ can be $\partial \Delta^4.$  For nonorientable $K,~\Delta$ can be appropriate connected sums of the $6$-vertex projective plane.  Similarly, for orientable $K$ other than the sphere, $\Delta$ can be appropriate connected sums of the $7$-vertex torus.   Now $\Sigma_v \Delta$ is a triangulation of $X$ with relatively minimal $g_2.$  
 
 For the induction step, let $C'$ be the compression body  obtained after attaching $m-1$ handles to $K \times [0,1]$ and let $X'$ be the corresponding pseudomanifold obtained by coning off each boundary component.     There are two possibilities, $C'$ and $X'$ are connected or $C'$ and $X'$ have two components.  
 
 Suppose $C'$ is connected. The location of the last handle on the top of $C'$ does not change the homeomorphism type of $C$ or $X.$ There is an orientation restriction if the top of $C'$ is an orientable surface.   The induction hypothesis gives us $\Delta'$ a triangulation of $X'$ which has  relatively minimal $g_2$ with respect to $v,$  where $v$ is the cone point for the top of $C'.$  A correctly oriented vertex folding at $v$ which involves no other singular vertices will provide the required triangulation.  If there are no such pairs of tetrahedrons with a correctly oriented admissible bijection, then sufficiently many facet subdivisions of facets containing $v$ can provide the needed tretrahedrons. 
 
 The situation when $C'$ has two components is similar.  Let $C'_1$ and $C'_2$ be the two components of $C'.$  Both components are compression bodies with fewer $1$-handles attached to their tops.  Any boundary connected sum of the tops of the two components results in a compression body homeomorphic to $C.$  The induction hypothesis gives us $\Delta'_1,$  a triangulation of  the pseudocompression body associated to $C'_1$ which has relatively minimal $g_2$ with respect to $v'_1$ and $\Delta'_2,$   a triangulation of  the pseudocompression body associated to $C'_2$ which has relatively minimal $g_2$ with respect to $v'_2.$  A connected sum of the form $\Delta'_1 \#_\psi \Delta'_2,$ where $\psi$ is a bijection between facets such that $\psi(v'_1) = v'_2$ and no other singular vertices are involved produces the desired triangulation.    As before, if no such facets are available, then one can subdivide facets which contain $v'_1$ and/or $v'_2$ as needed.
  \end{proof}
  
When is a multiset of connected compact surfaces the multiset of singularities of a pseudocompression body?  Equivalently, what are the possible multisets of boundary components of a connected compression body.  The emptyset corresponds to three-manifolds.  The following characterization of such nonempty multisets follows from the classification of surfaces.

\begin{proposition} \label{pseudocompression singularities}
Let $\Ss=\{K_1, \dots, K_m\} $ be a nonempty multiset of connected compact surfaces with  $b_1(K_1)\ge \cdots \ge b_1(K_n).$  Then there exists a  triangulation $\Delta$  of a pseudocompression body, such that $\Ss(\Delta)=\Ss$ if and only if

\begin{itemize}
  \item   $\displaystyle\sum^m_{i=1} b_1(K_i)$ is even.
  \item  If $K_1$ is orientable, then all $K_i$ are orientable.
  \item   $b_1(K_1) \ge \displaystyle\sum^m_{i=2} b_1(K_i).$
  \item  If $K_1$ is not orientable, then  $b_1(K_1) > \displaystyle\sum_{K_i \mbox{ orientable}} b_1(K_i).$
\end{itemize}
\end{proposition}

As we will see in the discussion after Theorem \ref{two singularities - homeomorphism} three-dimensional complexes with relatively minimal $g_2$ need not be pseudocompression bodies.  This still leaves the following problem.

\begin{problem}
Suppose $\Delta$ is three-dimensional, has relatively minimal $g_2,$  and $\Ss(\Delta)$ satisfies Proposition \ref{pseudocompression singularities}.  Is $\Delta$ a pseudocompression body?
\end{problem}
  
  One operation which never produces a complex with relatively minimal $g_2$ is handle addition.  

\begin{proposition} \label{no handle addition}
  Let $\Delta$ be a  normal $d$-pseudomanifold, $d \ge 3.$  Suppose $\psi$ is an admissible bijection between facets of $\Delta.$   Then $\Delta^\psi$ does not have relatively minimal $g_2.$
\end{proposition}

\begin{proof}
Suppose $g_2(\Delta^\psi) = g_2(\lk v).$ If $v$ is not in the domain or codomain of $\psi,$ then $g_2(\lk(v, \Delta^\psi)=g_2(\lk (v,\Delta)) \le g_2(\Delta) < g_2(\Delta)+\binom{d+1}{2} = g_2(\Delta^\psi).$  Hence $\Delta^\psi$ can not have relatively minimal $g_2$ with respect to $v.$  Alternatively, suppose $\bar{v}$ is the identified vertex in $\Delta^\psi$ of $v$ and $\psi(v).$  So $g_2(\lk (\bar{v},\Delta^\psi))=g_2(\lk (v,\Delta))  + g_2(\lk (\psi(v),\Delta)).$  Since $\psi$ is admissible, $\Star v$ and $\Star \psi(v)$ are disjoint.  Now the cone lemma implies that the dimension of the stress space for generic $f:V(\Delta) \to \RR^{d+1}$  is at least $g_2(\lk v) + g_2(\lk \psi(v)).$  Thus $g_2(\Delta^\psi) = g_2(\Delta)+ \binom{d+1}{2}~ \ge  g_2(\lk v) + g_2(\lk \psi(v))+\binom{d+1}{2} >g_2(\lk v) + g_2(\lk \psi(v)) = g_2(\lk \bar{v}).$
\end{proof}

 The most important property of a complex $\Delta$ with relatively minimal $g_2$ with respect to $\sigma$  is that it follows immediately from the cone lemma and Proposition \ref{stresses=g2} that all of its generic stresses must have support contained in $\Star \sigma.$  
The next two  lemmas use this and will allow us to reduce the problem of studying three-dimensional complexes with relatively minimal $g_2$  to complexes whose graph is contained in the star of a vertex.     The citations for these two lemmas are not strictly correct as there are no statements that say the same thing in Kalai's paper.  However, the proofs are so close to those in Section 9 of \cite{Kalai} we felt that we would be remiss not to point this out.

\begin{lemma} \label{lemma:no empty circles} \cite[Section 9]{Kalai}
Let $\Delta$ be a complex with relatively minimal $g_2$ with respect to $\sigma. $   If $\Omega$ is an induced subcomplex of $\Delta$ which is a circle, then $\Omega \subseteq \lk( \sigma).$
\end{lemma}

\begin{proof}
    Suppose $\Omega$ is not contained in the link of $\sigma.$  Then there is an edge $e$ of $\Omega$ such that $e$ is not in $\Star \sigma.$  Let $R$ be the union of the closed stars of the vertices of $\Omega$ not incident to $e.$  By lemma \ref{union lemma} $R$ is $d$-rigid, $d = \dim \Delta +1,$ and as previously noted, this implies $R \cup \{e\}$ is $d$-rigid. Since $\Omega$ is an induced subcomplex, $e \notin R$ and $g_2(R \cup \{e\})=g_2(R)+1.$  Hence in $R$ and $\Delta$ there is a generic stress whose support contains $e$, a contradiction.  
\end{proof}

\begin{lemma}\label{lemma:g2=sing} \cite[Section 9]{Kalai}
Suppose $\Delta$ has relatively minimal $g_2$ with respect to $\sigma.$   
If a vertex $v$ is not in $\Star (\sigma),$ then $\lk(v)$ is a stacked sphere. 

\end{lemma}

\begin{proof}
  
 We consider the cases of $\dim \Delta \ge 4$ and $\dim \Delta = 3$ separately.  First suppose $\dim \Delta \ge 4.$   Let $v$ be a vertex which is not in $\Star (\sigma).$  If $g_2(\lk v) \neq 0,$ then by the cone lemma there is a generic stress in $\Star v$ whose support includes an edge incident to $v.$   The support of such a stress is not contained in $\Star \sigma$.  Hence $g_2(\lk v)=0$ and the link of $v$ is a stacked sphere \cite{Kalai}.  
  
  Now we assume the dimension of $\Delta$ is three.  So $\sigma$ is a vertex which we now call $u.$  Since the links of singular vertices of $\Delta$ have positive $g_2,$  the same reasoning as above shows that $v$ is nonsingular.  To show that the link of $v$ is a stacked sphere it is sufficient to prove that the link of $v$ has no induced subcomplex $\Omega$  which is a circle of cardinality four or more.   So assume that the $\Omega$ is an induced subcomplex of the link of $v$  which is a circle of size four or more.  By Lemma \ref{lemma:no empty circles} $\Delta$ has no empty triangles which contain $v.$  Hence $\Omega$ is also an induced subcomplex of $\Delta$ and is contained in the link of $u.$  However, if we let $w,x$ be two vertices of $\Omega$  which do not share an edge, the induced subcomplex on $\{v,w,u, x\}$ is a square contradicting Lemma \ref{lemma:no empty circles}. 
\end{proof}

\begin{theorem} \label{no outside vertices}
  Suppose $\Delta$ is a $3$-pseudomanifold with relatively minimal $g_2$ with respect to $u.$    Then there exists $\Delta'$ and a vertex $u' \in \Delta'$ such that $\Delta$ is obtained from $\Delta'$ by (possibly zero) facet subdivisions and $G(\Delta')= G(\Star u').$
  \end{theorem}

\begin{proof}
The proof is by induction on the number of vertices of $\Delta$ not contained in the star of $u.$  For the base case of zero, let $e$ be an edge of $\Delta$ not contained in $\Star u.$  Then $\Star u \cup \{e\}$ is rigid and $g_2(\Star u \cup \{e\}) = g_2(\Star u)+1.$  Hence there are stresses which contain $e$ in their support and this is impossible.   

 For the induction step, let $v$ be a vertex of $\Delta$ not contained in the star of $u.$ By the previous lemma the link of $v$ is a stacked sphere.    There are now two possibilities.  One, the link of $v$ is the boundary of a $3$-simplex, say $wxyz.$  If $wxyz$ is a face of $\Delta,$ then the boundary of the four-simplex $vwxyz$ is in $\Delta$ and hence $\Delta$ is the boundary of the $4$-simplex.   If $wxyz$ is not a face of $\Delta,$ then  we can remove $v$ and insert the tetrahedron $wxyz$ for the induction step.   So now we assume the link of $v$ is any other stacked sphere and use an approach originally introduced by Walkup \cite[Lemma 10.4]{Walkup}.  The link of $v$ has at least one missing triangle, say  $xyz.$  We first observe that $xyz$ is a face of $\Delta.$  Otherwise, we could retriangulate the $3$-ball $\Star v$ by removing $v$, inserting $xyz$ and then coning off the two spheres formed by the link of $v$ and $xyz.$  The resulting complex would have $g_2$ one less that $\Delta,$ but would still have a vertex whose link is the same as the link of $u.$    Since this is impossible, $xyz$ is in $\Delta$ and $vxyz$ is a missing tetrahedron.  
 
 The next step is to show that if at most two of $\{x,y,z\}$ are singular, then $\Delta = \Delta_1 \# \Delta_2$ where $v \in \Delta_1.$  When $\Delta$ can be written in this  form $g_2(\Delta_1) = g_2(\Delta)$ and   (\ref{g_2:connected sum}) shows that $\Delta_2$ is a stacked sphere.  Induction on the number of vertices in a stacked sphere show that $\Delta = \Delta_1 \# \Delta_2,$ with $\Delta_2$ a stacked sphere, if and only if $\Delta$ can be obtained from $\Delta_1$ by repeatedly subdividing facets    and  the induction step is complete.  The last step is to prove that it is impossible for $x,y$ and $z$ to be singular for every missing tetrahedron $vxyz$ with $v \notin \Star u.$ 

 Suppose at most one of $x,y$ and $z$ is singular.    Then Lemmas \ref{connected sum and handle addition}, \ref{lemma:missingtetra2} and Proposition \ref{no handle addition}   imply that $\Delta$ was formed via connected sum as promised or a vertex folding at $x,y$ or $z.$    If $\Delta'$ is a vertex unfolding of $\Delta,$ then $\Delta'$ contains a vertex whose link is isomorphic to the link of $u.$ So  (\ref{folding g2})  eliminates  vertex folding as a possibility.  
 
 So now assume exactly two of $\{x,y,z\}$ are singular, say $x$ and $y.$  If the missing triangle $\partial vyz$ in the link of $x$ or the missing triangle $\partial vxz$ in the link of $y$ separate their respective links, the argument of the previous paragraph applies.  So suppose neither missing triangle separates the links of $x$ and $y.$    By Lemma \ref{lemma:missingtetra1} small neighborhoods of $\partial vyz \subseteq \lk x$ and $\partial vxz \subseteq \lk y$ are either both annuli or M\"obius strips.  If both are M\"obius strips, then  Lemma \ref{lemma:missingtetra3} gives us a $\Delta'$ an edge unfolding of $\Delta$  which contains a vertex whose link is isomorphic to the link  of $u.$  So (\ref{edge folding g2}) eliminates this possibility.  Lastly, we consider the possibility that both neighborhoods are annuli.  From the discussion following Lemma \ref{lemma:missingtetra3}  there is an edge unfolding $\Delta'$ which is not a normal pseudomanifold as two of its vertex links are surfaces with a pair of vertices identified.  However, surfaces with two vertices identified are minimal cycles and Fogelsanger's thesis \cite{Fogelsanger} implies they are $3$-rigid.  Hence Kalai's original proof of the rigidity of normal pseudomanifolds, \cite{Kalai} applies verbatim to $\Delta'.$  This allows us to repeat the argument applied to a pair of M\"obius strips and show that two annuli which do not separate their links is also impossible.   For an alternative approach to analyzing when both neighborhoods are annuli see Remark \ref{two annuli} below.
 
For the last step we assume that for all $v \notin \Star u$ and missing tetrahedron $vxyz,$ the vertices $x,y,$ and $z$ are singular.  In particular, they are in the link of $u.$    We begin by observing that $G(\lk v \cap \lk u)$ is a complete graph.  To see this, let $w_1$ and $w_2$ be distinct vertices  of $G(\lk v \cap \lk u).$  The edge $w_1 w_2$ is in $\Delta.$  Otherwise, $\{v,w_1, u, w_2\}$ is an induced circle which is not contained in the star of $u.$  On the other hand, $w_1 w_2$ must be in  $\Star u$ since if not,~$g_2(\Star u \cup \{w_1 w_2\}) = g_2(\Star u)+1 = g_2(\Delta) +1.$ Since the link of $v$ is a sphere, the two possibilities for $G(\lk v \cap \lk u)$ are the empty triangle $\partial xyz$ or a $K_4$ with vertices $\{x,y,z,w\},$ where $w \in \lk u.$  

Now consider the link of $v.$  The empty triangle (in $\lk  v$ - it is not empty in $\Delta$) $\partial xyz$ separates the link of $v$ into two components.  Only one of these two components may contain another vertex of $\lk u.$  If a component of $\lk v \setminus \partial xyz$ which does not contain a vertex of the link of $u$ has two (or more)  vertices, then $\lk v$ has an empty triangle with at most two singularities - a situation dealt with above.  Hence there exists $t$ in the complement of the link of $u$ such that $xyt, xzt$ and $yzt$ are triangles in the link of $v.$  The same arguments applied to $v$ also apply to $t.$ Designating $v_0=v$ and $v_1=t$ we obtain a minimal sequence of vertices $(v_0, v_1, \dots, v_m)$ such that 
 \begin{itemize}
   \item For  all $0 \le i \le m,~ \partial xyz$ is a missing triangle in the link of $v_i$
   \item
     $v_i \notin \Star u, 0 \le i \le m-1.$
    \item 
     $v_m \in \lk u$ or $v_m = v_i$ for $i<m.$
     \item
      For $0 \le i \le m-1$ the link of $v_i$ contains the triangles $xyv_{i+1}, xzv_{i+1}$ and $yzv_{i+1}.$
 \end{itemize}
 
Suppose $v_i = v_m$ for some $i<m.$  Then $ \bigcup^m_{j=i} \Star v_j$  is the join of $\partial xyz$ and the circle $C(v_i, v_{i+1}, \dots, v_{m-1}, v_m).$  Since this is a triangulation of $S^3$ we reach a contradiction.  So $v_m$ is in the link of $u.$

On the  side of $\partial xyz$ in the link of $v$ which does not contain $t$ we have a similar sequence of vertices.  Putting these two sequences together we have a relabeled sequence of  vertices $(s_0, s_1, \dots, s_{j-1}, s_j)$ such that 

\begin{enumerate}
  \item  Except for possibly $s_0=s_j,$ the vertices are distinct. 
  \item $s_0$ and $s_j$ are in the link of $u.$  All other $s_i$ are outside of the star of $u.$  
  \item For $0 \le i \le j-1,~ s_i$ is in the link of $s_{i+1}.$
  \item For $2 \le i \le j-2,$ the link of $s_i$ is $\partial xyz \star \partial s_{i-1} s_{i+1}.$  The link of $s_1$ contains the graph $G(\partial xyz \star \partial s_0 s_2).$  Similarly, the link of $s_{j-1}$ contains $G(\partial xyz \star \partial s_{j-2} s_j).$  
\end{enumerate}
 
 Let $G_i = G(\partial xyz \star \partial s_{i-1} s_{i+1}).$ For $1 \le i \le j-1, G_i$ is a subcomplex of $\Delta.$  Now consider $\Omega= (\Star u) \cup  \bigcup^{j-1}_{i=1} G_i.$ Repeated applications of the union lemma show that $\Omega$ is $4$-rigid and is a subcomplex of $\Delta.$   However, $g_2(\Omega) = g_2(\Star u) +1= g_2(\Delta)+1.$  This violates Lemma \ref{g2 subcomplex < g2 complex} and provides the contradiction required to complete the proof.
\end{proof}

\begin{remark} \label{two annuli}
  When analyzing the case of the two neighborhoods of missing triangles being annuli which do not separate their links, there is an alternative approach which ignores the links of  $x$ and $y$ in $\Delta'$ the edge unfolding of $\Delta.$  One can use the $4$-connectivity of $\Delta,$ \cite{Barnette3}, to show that the graph of the edge unfolding of $\Delta$ is still connected after removing $x$ and $y.$ This allows one to apply \cite[Proposition 6.4]{Kalai} to establish the $4$-rigidity of $\Delta'.$
   \end{remark}

\section{Few singularities}  \label{few singularities}

What topological restrictions does relatively minimal $g_2$ put on a complex $\Delta$ or $\Ss(\Delta)?$  It already follows from Lemma \ref{lemma:no empty circles} that $\Delta$ must be simply connected.  In this section we describe the homeomorphism type of all three-dimensional complexes with relatively minimal $g_2$ and at most two singularities.

\begin{theorem} \label{two singularities - combinatorics}
Let  $\Delta$ be a normal $3$-pseudomanifold with relatively minimal $g_2.$ If $\Delta$ has at most two singularities, then  $\Delta$ can be obtained from the boundary of a $4$-simplex by a sequence of operations of types $(1),\dots,(5)$ of Proposition \ref{relatively minimal preserving operations}  and any application of (\ref{relatively minimal via connected sum}) involves at most one singular vertex. 
\end{theorem}

\begin{proof}
If $\Delta$ has no singularities, then $\Delta$ is  a triangulation of  a closed $3$-manifold.  So the link of every vertex is a sphere and has $g_2=0.$  Hence $g_2(\Delta)=0$ and $\Delta$ is a stacked sphere \cite{Walkup}.   

Suppose $\Delta$ has one singular vertex $v$ and $\Delta$ has relatively minimal $g_2$ with respect to $v.$ We proceed by induction on  $b_1(\lk v)$ and, when necessary, the number of vertices of $\Delta.$    The minimal triangulations of the solid torus and solid Klein bottle with their boundaries coned off are the base cases.     Combining Proposition \ref{no outside vertices} and operation (2) of Proposition \ref{relatively minimal preserving operations} we can assume that $G(\Delta) = G( \Star v).$  Let  $\tilde{\Delta}$ be the induced subcomplex on $V - \{v\}.$ So $\tilde{\Delta}$  is   a triangulation of  a $3$-manifold with boundary the link of $v.$  In addition, our assumption implies that $\tilde{\Delta}$ has no interior edges.  Let $xyz$ be an interior triangle of $\tilde{\Delta}.$  Then $vxyz$ is a missing tetrahedron in $\Delta.$  Depending on whether or not $\partial xyz$ separates the link of $v$, $\Delta$ is either a connected sum or a vertex folding at $v.$   So $\Delta$ is obtained by applying operations (3) or (4) of Proposition \ref{relatively minimal preserving operations} to complexes which satisfy the induction hypotheses.  

Finally, assume that $\Delta$ has two singular vertices $v_1$ and $v_2.$ Let $\Phi_1$ be the link of $v_1$ and $\Phi_2$ the link of $v_2.$  Without loss of generality we can assume that  $b_1(\Phi_1) \ge b_1(\Phi_2)$.  Thus $\Delta$ has relatively minimal $g_2$ with respect to $v_1.$    We proceed by induction on  $b_1=b_1(\Phi_1)-b_1(\Phi_2)$ and, when necessary, the number of vertices in $\Delta.$  The base case is  $b_1=0.$ As was the case for one singularity, we can assume that $G(\Delta) = G(\Star v_1).$ Hence $v_1 v_2$ is an edge of $\Delta.$  If every triangle $xyz$ in the link $v_2$ which does not contain $v_1$ is also a triangle in the link of $v_1,$ then $\Delta$ is the one-vertex suspension of the link of $v_1,$ a single application of operation (5). So assume that $xyz$ is a triangle in $\lk v_2$ but is not a triangle in $\lk v_1.$  This implies that $v_1xyz$ is a missing tetrahedron in $\Delta.$  As before this implies that $\Delta$ either comes from a vertex folding at $v_1$ or a connected sum where $v_1xyz$ is one of the identified facets which has been removed. The usual formulas for $g_2$ and the links of foldings and connected sum show that a folding is impossible. Thus $\Delta$ is  a connected sum, $\Delta = \Delta_1 \#_\psi \Delta_2,$ where one of the $\Delta_i$ is a stacked sphere.  So the induction hypothesis on the number of vertices applies to both $\Delta_i.$ 

   For the induction step we can assume that  $b_1 >0.$ Now, the exact same reasoning as when  $b_1=0$ shows that $\Delta$ was obtained via connected sum or vertex folding and in both cases involves complexes with smaller  $b_1$, number of vertices, or number of singularities.  
\end{proof}

\begin{theorem} \label{two singularities - homeomorphism}

Let $\Delta$ be a normal $3$-pseudomanifold with at most two singularities.  Then $\Delta$ is a  triangulation of a pseudocompression body.

\end{theorem}

\begin{proof}
If $\Delta$ has no singularities, then $\Delta$ is a stacked sphere and hence  $|\Delta|$ is homeomorphic to $S^3.$  The $3$-sphere is a pseudocompression body since it is $S^2 \times [0,1]$ with both boundary components coned off.

Suppose $\Delta$ has one singular vertex, say $v.$  We claim $\Delta$ is   a triangulation of  a handlebody with its boundary coned off.  This claim would prove that $\Delta$ is   a triangulation of  a pseudocompression body since such a space is $S^2 \times [0,1]$ with $1$-handles attached to $S^2 \times \{1\}$ and both boundary components coned off.  From the proof of the previous theorem $\Delta$ can be obtained using operations (1)-(4) of Proposition \ref{relatively minimal preserving operations}.  Induction on the number times (3) or (4) is used and, if necessary, the number of vertices, is sufficient to prove the claim.  When (3) is invoked the induction hypotheses tell us that $\Delta = \Delta_1 \#_\psi \Delta_2$ where $\Delta_1$ has relatively minimal $g_2$ with respect to $v_1,~ \Delta_2$ has relatively minimal $g_2$ with respect to $v_2$ and both $H_1=\Delta_1 \setminus v_1$ and $H_2=\Delta_2 \setminus v_2$ are    triangulations of  handlebodies.  Thus $\Delta \setminus v$ is   a triangulation of  the handlebody obtained by taking a boundary connected sum of $H_1$ and $H_2.$   When (4) is used the induction hypotheses tell us that  $\Delta = \tilde{\Delta}^\psi_v,$ where $\tilde{\Delta}$ has relatively minimal $g_2$ with respect to $v$ and $\tilde{\Delta} \setminus v$ is   a triangulation of  a handlebody $\tilde{H}.$  In this case $\Delta \setminus v$ is   a triangulation of  the handlebody $\tilde{H}$ with two triangles on its boundary identified and hence is also   a triangulation of  a handlebody.

Finally, assume $\Delta$ has two singularities, $v$ and $w,$ and $\Delta$ has relatively minimal $g_2$ with respect to $v.$   As above, we can assume that $G(\Delta)=G(\Star v).$  We induct on the number of operations (1)-(5) used to construct $\Delta$  as in the previous theorem.  Since all connected sums involve a complex with at most one singular vertex we can always start with the one-vertex suspension of a surface at a graph cone point.  So, after the first operation we have the suspension of a surface and hence a pseudocompression body.  It also follows from the proof of Theorem \ref{two singularities - combinatorics} that the link of $w$ is never altered by any subsequent operation.  

The compression body decomposition of $\Delta \setminus v$ is part of the induction hypotheses.  Stellar subdivide the edge $vw$ and call the new  homeomorphic complex $\Delta'.$ Since the link of $w$ never changes, $\lk (w,\Delta')$ is now $\lk (w,\Delta)$ with the new vertex replacing $v.$   Now remove a small open ball $B(w)$ around $w$ from $|\Delta'|.$  This leaves $|\Star w| \setminus B(w)$ homeomorphic to $|\lk w| \times [0,1]$ after every induction step.  Deleting $v$ (and its incident faces) leaves a compression body whose top is $|\lk (v,\Delta|)$  when $\Delta$ is the one-vertex suspension of a surface.  It remains to check that all of this holds after an application of (3) or (4) to a complex $\Delta$ which satisfies the induction hypotheses.  Let $C$ be the compression body $|\Delta'\setminus v| - B(w).$  Executing (3) is the same as   performing a boundary connected sum with a handlebody on the top of $C$ and reconing off the  boundaries.  A vertex folding involving one singular vertex is the same  identifying two triangles on the top of $C$,  thus adding a single handle to the top of $C$, and reconing off the boundaries.  
\end{proof}

\noindent In terms of understanding the minimum $g_2$ of three-dimensional normal pseudomanifolds with two singularities Theorems \ref{two singularities - combinatorics} and \ref{two singularities - homeomorphism} leave open the following problem.  

\begin{problem}
Let $K_1$ and $K_2$ be connected compact surfaces such that $K_1$ is orientable, $K_2$ is not orientable and  $b_1(K_1) \ge b_1(K_2).$    What is $\Gamma(\{K_1,K_2\})?$
\end{problem} 

\noindent Using connected sums an upper bound for the answer to this problem is  $3 \cdot (b_1(K_1) + b_1(K_2)).$  There are no smaller $g_2$ among all examples with $10$ or fewer vertices \cite{AkhShiSwa}.

We do not know if there are any $3$-dimensional complexes with relatively minimal $g_2$, three or four singularities, but are not pseudomcompression bodies. Datta and Nilakantan determined all three-dimensional normal pseudomanifolds on eight vertices \cite{DattaNilakantan}.  One of them, which they denoted N3, has five singularities, four projective planes and one torus.  It has relatively minimal $g_2$ but does not satisfy Proposition \ref{pseudocompression singularities} and hence is not a pseudocompression body.  An exhaustive computer search shows that all sets of three or four singularities which have a   triangulation with $10$ or fewer vertices and relatively minimal $g_2$ also satisfy the hypotheses of Proposition \ref{pseudocompression singularities} \cite{AkhShiSwa}.  

\begin{problem}
Are there $3$-dimensional complexes with relatively minimal $g_2$,   three or four singularities, but  are not pseudocompression bodies?
\end{problem}

\section{$g_2 = 3$}

  Walkup proved that triangulations of $3$-manifolds with $g_2=0$ are stacked spheres \cite{Walkup}.  This was extended to all $d$-manifolds, $d \ge 3,$  by Kalai \cite{Kalai}.  Fogelsanger's thesis \cite{Fogelsanger}  proved that this result also holds for all normal $d$-pseudomanifolds, $d \ge 3.$  Later, Nevo and Novinsky showed that homology $d$-manifolds with $g_2=1$ are polytopal and gave an explicit description of their combinatorial structure \cite{NevoNovinsky}.  Induction on dimension and Lemma \ref{g2>face} applied to  links extends this result to all normal pseudomanifolds.   In her 2018 paper \cite{Zheng} Zheng repeated this feat for $g_2=2.$  She also noted that this is the last possible result of this nature.  Specifically, for all $i \ge 0$ there exist triangulations of $S^i \ast \RR P^2$ with $g_2=3.$  Zheng asked if  $S^d$ and  $S^{d-3} \ast \RR P^2$ are the only possible topological types for normal $d$-pseudomanifolds with $g_2=3.$  The purpose of this section is to give an affirmative answer to this question in dimension three and provide a simple characterization of such complexes.  
  
Let $\Psi$ be a triangulation of $\RR P^2$ that has a graph cone point $v.$    Now let $\Delta_{\Psi, v}$ be obtained by performing an arbitrary number, possibly zero, of facet subdivisions of $\Sigma_v \Psi.$ Since $v$ is a graph cone point of $\Psi,~g_2(\Sigma_v \Psi)=3.$ Facet subdivisions do not alter $g_2,$ so $g_2(\Delta_{\Psi,v})=3$ and $|\Delta_{\Psi,v}|$ is homeomorphic to  $S^{0} \ast \RR P^2$. The result of this section is the converse - any $3$-dimensional normal pseudomanifold  with $g_2=3$ which is not  a triangulation of the $3$-sphere, is $\Delta_{\Psi,v}$ for some $\Psi$ and $v.$  

 \begin{theorem} \label{g2=3 complexes}
 Let $\Delta$ be a $3$-dimensional normal pseudomanifold which is not a triangulation of $S^3.$  If $g_2(\Delta)=3,$ then $\Delta$ is obtained by taking the one-vertex suspension of a triangulation of $\RR P^2$ at a graph cone point and subdividing facets.
  \end{theorem}  
   
 \begin{proof}  As noted above, if all of the vertices of $\Delta$ are nonsingular,  then $\Delta$ is a triangulation of the $3$-sphere.  By Lemma \ref{g2>face} the links of all of the singular vertices of $\Delta$ are  triangulations of $\RR P^2.$  The sum of the Euler characteristics of any $3$-dimensional normal pseudomanifold is even, so there are an even number of vertices with $\RR P^2$ links.    
  
  The difference $h_3 - h_1$ in a $3$-dimensional normal pseudomanifold is $\displaystyle\sum_{v \in V} (2 - \chi(\lk v)).$   See, for instance, \cite{NovikSwartz}.  By \cite[Corollary 1.8]{Swartz2014} $g_2(\Delta) = 3$ implies that $g_3 \le 4.$  Since $h_3-h_1 = g_2 + g_3$ and all of the singular vertices have Euler characteristic one, the number of singular vertices of $\Delta$ is less than eight.  Theorem 2.12 of \cite{Swartz2014}  implies that if $\Delta$ has six singular vertices, then $g_2(\Delta) \ge 5$ and if it has four singular vertices, then $g_2(\Delta) \ge 4.$  Thus $\Delta$ has two singular vertices and the links of both of these vertices  are triangulations of $\RR P^2.$  We conclude that $\Delta$ has relatively minimal $g_2$ and Theorem \ref{two  singularities - homeomorphism} completes the proof.
  \end{proof}
  
 There are at least ten different homeomorphism types of $3$-dimensional normal pseudomanifolds which can be triangulated with $g_2=6$ \cite{AkhShiSwa}.  However, the following is unknown at present.
 
 \begin{problem}
Is there a $3$-dimensional normal pseudomanifold $\Delta$ such that the minimum of $g_2$ over all triangulations homeomorphic to $\Delta$ is $4$ or $5$?
 \end{problem}  
 
 \bigskip

\noindent {\bf Acknowledgement:} The authors would like to thank the anonymous referees for many useful comments and suggestions. This work has been done when the first author was an SERB Indo-US Postdoctoral Fellow at Cornell University, and currently the first author is supported by DST INSPIRE Faculty Research Grant (DST/INSPIRE/04/2017/002471).

{

\end{document}